\documentclass[preprint,1p]{elsarticle} 
\usepackage{amssymb,amsmath,amsthm,dsfont}
\usepackage{dsfont}

\usepackage{here}
\usepackage{graphicx}
\usepackage{tabularx}
\usepackage{arydshln}
\usepackage{bbm}
\usepackage[usenames,dvipsnames]{xcolor}

\usepackage{booktabs, multirow}
\usepackage[font={scriptsize}]{caption}

\usepackage{placeins} 

\newcommand{\mbf}[1]{\mathbf{#1}}
\newcommand{\gaus}[1]{\lfloor #1 \rfloor}
\newcommand{\undnr}[1]{\underline{\mathbf{#1}}}
\newcommand{\Z}{\mathbb{Z}}

\newcommand{\R}{\mathbb{R}}
\newcommand{\N}{\mathbb{N}}
\newcommand{\var}{\mathrm{var}}
\newcommand{\Cov}{\mathrm{cov}}
\newcommand\mycom[2]{\genfrac{}{}{0pt}{}{#1}{#2}}

\newcommand{\btheta}{\boldsymbol{\theta}}
\newcommand{\bgamma}{\boldsymbol{\gamma}}
\newcommand{\blambda}{\boldsymbol{\lambda}}
\newcommand{\sprod}[1]{\left\langle #1 \right\rangle}

\newtheorem{theorem}{Theorem}
\newtheorem{corollary}{Corollary}
\newtheorem{lemma}{Lemma}
\newtheorem{remark}{Remark}
\newtheorem*{remark*}{Remark}
\newtheorem{defi}{Definition}

\journal{Journal of Multivariate Analysis}

\begin{document}

\begin{frontmatter}



\title{Change-Point Detection and Bootstrap for Hilbert Space Valued Random Fields}


\author[bb]{B\'eatrice Bucchia\corref{cor}} 
\ead{bbucchia@math.uni-koeln.de}
\author[mw]{Martin Wendler}
\ead{martin.wendler@uni-greifswald.de}

\cortext[cor]{Corresponding author.}
\address[bb]{Mathematical Institute, University of Cologne, Weyertal 86-90, 50931 K\"oln, Germany}
\address[mw]{Department of Mathematics and Computer Science, Ernst Moritz Arndt University Greifswald, Walther-Rathenau-Stra\ss e 47, 17487 Greifswald, Germany}

\begin{abstract}
The problem of testing for the presence of epidemic changes in random fields is investigated. In order to be able to deal with general changes in the marginal distribution, a Cram\'er-von Mises type test is introduced which is based on Hilbert space theory. A functional central limit theorem for $\rho$-mixing Hilbert space valued random fields is proven. In order to avoid the estimation of the long-run variance and obtain critical values, Shao's dependent wild bootstrap method is adapted to this context. For this, a joint functional central limit theorem for the original and the bootstrap sample is shown. Finally, the theoretic results are supplemented by a short simulation study.
\end{abstract}

\begin{keyword}
change-point detection \sep dependent wild bootstrap \sep FCLT for Hilbert space valued r.v. \sep random fields 
\MSC[2010]
62H15 
\sep 62E20 
\sep 62M99 
\sep 60G60 

\end{keyword}

\end{frontmatter}

\begin{section}{Introduction}

\subsection{Change-point tests for random fields}
The focus of this paper lies on the problem of epidemic change in the mean for Hilbert space valued random fields.  
Given a data set of observations, a classical problem in change-point analysis consists of testing whether all the observations have the same stochastic structure (i.e. marginal distribution) or whether there is a subset (the change-set) of the data where the structure is different. 
For data corresponding to a time series, the split into different data subsets can be characterized by the points in time (the change-points) at which there is a structural break. In the epidemic change model, there are two possible change-points (the start and end of an ``epidemic'') and the structure of the data changes after the first change-point but reverts back to its original state after the second change-point. Extended to random fields, this becomes the problem of testing for rectangular change-sets. 
Epidemic changes are of interest not only in medicine (see, e.g., \cite{levinKline}) but also, e.g., in signal detection and textile fabric quality control (see, e.g., \cite{zhangBresee}). The epidemic change-point problem was introduced by Levin and Kline \cite{levinKline} and has since been the subject of numerous publications (see, e.g., \cite{astonKirch}, \cite{csorgoHorvath}, \cite{linJaruskova}, \cite{rackauskas} and the publications listed therein). For random fields with a change in the mean, a nonparametric approach for this type of problem was considered in \cite{jaruskovaPiterbarg} and \cite{zemlysPHD} for i.i.d. observations and in \cite{bucchia2014} and \cite{bHeuser} for weakly dependent data. The test statistics considered in these publications are a special type of scan statistic, variants of which could - under the assumption that the distributions of the observations belong to a parametric family - also be used to test for changes in other parameters of a distribution (see, e.g., \cite{jaruskovaPiterbarg}, \cite{loader1991}, \cite{siegmund2000}). For the nonparametric problem of a change in the distributions without any prior information on the family of distributions, however, a test based on the empirical distribution function $F_n$ with 
\begin{equation*}
F_n(t)=\frac{1}{n}\sum_{i=1}^n\mathbbm{1}_{\{X_i\leq t\}}
\end{equation*}
might be more useful. Equipped with the appropriate norm, one can regard these as sums of Hilbert space valued random variables, where the true distribution function of $X_i$ is the expected value (in the Hilbert space) of $\mathbbm{1}_{\{X_i\leq \cdot\}}$. Therefore, the change in distribution problem can be translated into a change in mean problem for Hilbert space valued random variables. 

The analysis of functional data over a spatial region is of independent interest. As a special case of spatio-temporal data, where measurements over time are taken at different locations in space, functional data may arise for instance in brain imaging or in space physics (see \cite{gromenkoKokoszka2012}). 

For weakly dependent time series of functional data, the epidemic change model was investigated by Aston and Kirch \cite{astonKirch}, who constructed test statistics based on projections on the principal components. By contrast, we aim to apply the approach used by Sharipov et al. \cite{sharipov2014}, who take the full functional structure into account. To the best of our knowledge, there are no results on asymptotic change-point tests for the specific setting considered here.

A popular approach for the construction of asymptotic tests for change in mean problems are so-called CUSUM-type tests, where the mean is estimated using cumulative sums of the observations. This leads to test statistics that can be written as functionals of the partial sum process of the data. Therefore, the first aim of this paper is to give a functional central limit theorem (FCLT) for weakly dependent Hilbert space valued random fields which can then be used for change-point tests. The continuous mapping theorem can then be applied to obtain the limit distribution of a CUSUM-type test statistic.

Although the central limit theorem is known for multivariate weakly dependent random fields (see \cite{bulinski2004}, \cite{tone2010central}) and even random fields with values in a Hilbert space (see \cite{tone2011central}), most of the literature on FCLTs for random fields has focused on real-valued fields. For this one-dimensional setting, numerous results have been given not only for independent observations (see \cite{wichuraInequalities}) but also for weakly dependent fields. For instance, the monographs by Bulinski and Shashkin \cite{bulinskiShashkin} and Lin and Lu \cite{LinLu} give examples of FCLTs under conditions related to association and mixing conditions respectively. For mixing fields of real-valued random variables, Deo \cite{deoFCLT1975, deoNote1976} proved FCLTs under $\varphi$-mixing conditions and Kim and Seok \cite{kimSeok1995} extended the ideas of Deo's proofs to obtain FCLTs for $\rho$-mixing random fields. For i.i.d. Hilbert space valued random fields, Zemlys \cite{zemlysPHD} introduced a H\"olderian FCLT. The FCLT presented here can be viewed as an extension of the approach by Deo \cite{deoFCLT1975} first to vector-valued fields and then to Hilbert space valued fields.

After describing the bootstrap method considered here (section 1.2), we introduce the notations used throughout this article (section 1.3). We then present our main results in section 2. To illustrate our theoretical findings, our third section reports some simulation results. Proofs of our main results are relegated to section 4.

\begin{subsection}{Bootstrap for Hilbert space valued processes}

Nonparametric resampling methods like bootstrap are especially useful when dealing with stochastic processes, as the asymptotic distribution typically depends on a parameter function, which is hard to estimate. The bootstrap of the empirical distribution function has been well studied, starting with \cite{bickel1981} in the independent case. This was extended to time series data by Naik-Nimbalkar and Rajarshi \cite{naik1994}, Peligrad \cite{peligrad1998} and Radulovi\'c \cite{radulovic2009} using block bootstrap methods adjusted for dependence. For an overview of the block bootstrap methods, see the book by Lahiri \cite{lahiri2003}. Shao \cite{shao2010} introduced a different resampling method for time series: the dependent wild bootstrap, which generalizes Wu's \cite{wu1986} wild bootstrap. Recently, Doukhan et al. \cite{doukhan2015} extended the dependent wild bootstrap to empirical distribution functions and were able to show its validity. 
As seen above, the empirical distribution function can be interpreted as a function of Hilbert space valued random variables. 

For more general Hilbert spaces, the bootstrap has been investigated in \cite{dehling2015} and \cite{politis1994}.

For the application to change-point detection, one needs a sequential bootstrap to mimic the behavior of the partial sum process. The consistency of the sequential multiplier bootstrap for the empirical distribution function under independence was shown in \cite{gombay1999} and \cite{holmes2013} for the sequential empirical process indexed by functions. For dependent data, Inoue \cite{inoue2001} proposed a block multiplier bootstrap for the sequential empirical distribution function. Sharipov et al. \cite{sharipov2014} studied block bootstrap for the partial sum process of Hilbert space valued random variables.

While there is a broad range of results for different bootstrap methods in the time series setting, much less work has been done for random fields, although ideas for this can be traced back thirty years to \cite{hall1985}. Politis and Romano \cite{politis1993} studied block bootstrap for partial sums, Zhu and Lahiri \cite{zhu2007} for the empirical distribution function. We are not aware of any bootstrap methods for Hilbert space valued random fields or of sequential bootstrap methods for the partial sum process of random fields (even in the real valued case).

The second aim of the paper is thus to give a sequential bootstrap method for Hilbert space valued random fields. We propose a generalization of the dependent wild bootstrap to random fields (a definition of the notation used in the following can be found in section 1.3): Let $(X_{\mathbf{k}})_{\mathbf{k}\in\Z^d}$ be a random field and $\bar{X}_{n}=\frac{1}{n^d}\sum_{\undnr{1}\leq\mathbf{i}\leq\undnr{n}}X_i$. Furthermore, let $\{V_{n}(\mathbf{i})\}_{\undnr{1}\leq\mathbf{i}\leq\undnr{n}}$ be a real valued random field, independent of $(X_{\mathbf{k}})_{\mathbf{k}\in\Z^d}$, with $\mathrm{E} V_{n}(\mathbf{i})=0$, $\operatorname{var}\{V_{n}(\mathbf{i})\}=1$ and a dependence structure to be specified later. The partial sum process $\{S_{n}(\mathbf{t})\}_{\mathbf{t}\in[0,1]^d}$ with
\begin{equation}
S_n(\mathbf{t})=n^{-d/2}\sum_{\undnr{1}\leq\mathbf{i}\leq\lfloor n\mathbf{t} \rfloor}\left(X_{\mathbf{i}}-\mu\right)
\label{eq:partialSum}
\end{equation}
will be bootstrapped by
\begin{equation}
S_n^\star(\mathbf{t})=n^{-d/2}\sum_{\undnr{1}\leq\mathbf{i}\leq\lfloor n\mathbf{t}  \rfloor}V_{n}(\mathbf{i})\left\{X_{\mathbf{i}}-\hat{\mu}(\mbf{i})\right\},
\label{eq:bootstrapDef}
\end{equation}
where $\hat{\mu}(\cdot)$ is an estimator for the mean function.

If the bootstrapped partial sum process mimics the behavior of the original partial sum process, by the continuous mapping theorem, the same holds for the bootstrap version of our test statistic. The classical choice proposed by Shao \cite{shao2010} for the mean estimator is $\hat{\mu} \equiv \bar{X}_n$. However, under the alternative (presence of a change), the bootstrap with this choice of estimator might not be close to the distribution under the null hypothesis (no change).
Therefore, we propose a different variant of our bootstrap. Let $\hat{C}_n$ be an estimator of the change-set such that $\varepsilon_1 n^d \leq\#\hat{C}_n\leq(1-\varepsilon_2)n^d$ for some \\$0<\varepsilon_1<1-\varepsilon_2<1$ and all $n\in\N$. Define
\begin{equation*}
\tilde{\mu}(\mathbf{k})=\begin{cases}
\frac{1}{\#\hat{C}_n}\sum_{\mathbf{i}\in\hat{C}_n}X_{\mathbf{i}}\ \ \ &\text{if }\mathbf{k}\in\hat{C}_n,\\
\frac{1}{\#\hat{C}_n^c}\sum_{\mathbf{i}\notin\hat{C}_n}X_{\mathbf{i}}\ \ \ &\text{if }\mathbf{k}\notin\hat{C}_n.
\end{cases}
\end{equation*}
In the following, we will consider bootstrapped versions of $\{S_{n}(\mathbf{t})\}_{\mathbf{t}\in[0,1]^d}$ with either of these two mean estimators, i.e., $\hat{\mu}$ will denote either $\bar{X}_n$ or $\tilde{\mu}(\cdot)$. We will not specify the change-set estimator $\hat{C}_n$, but assume that it is a subblock of $(\undnr{0},\undnr{n}]$ which fulfills the size restriction above (see \cite{bHeuser} for some example for $\R^p$-valued random fields). For an example of a change-set estimator, see our simulation study in section 3. 
\end{subsection}

\begin{subsection}{Notations} 
Before introducing the main results, we will now cover some notations and conventions that will be used throughout this paper. $\R^d$ denotes the vector space of real vectors, equipped with the usual partial order, and $\Z^d$ and $\N^d$ denote the subsets of integer and positive integer vectors, respectively. For an integer $k\in\Z,$ we denote $(k,\ldots,k)^\top\in\Z^d$ by $\undnr{k},$ and write general vectors $(x_1,\ldots,x_d)^\top\in\R^d$ as $\mbf{x}$. 
For $\mbf{x}\in\R^d$, we use the following notations: $\gaus{\mbf{x}}=(\gaus{x_1},\ldots,\gaus{x_d})^\top$ is the integer part of $\mbf{x}$, $|\mbf{x}|=(|x_1|,\ldots,|x_d|)$ and $[\mbf{x}]=x_1\cdots x_d.$ For a set $S\in\R^d$ and a number $n\in\N$, we write 
\[S\ominus S =\{\mbf{x}\in\R^d:\exists \mbf{s},\mbf{t}\in S,\, \mbf{x}=\mbf{s}-\mbf{t}\},\]
$\# S=\text{card}(S)$ if $S$ is finite, and $n S:=\{n \mbf{x}:\,\mbf{x}\in S\}$, where $n\mbf{x}=(n x_1,\ldots,n x_d)^\top$. 

A block in $\R^d$ is a set of the form \mbox{$(\mbf{x},\mbf{y}]=\{\mbf{z}:\, x_i < z_i \leq y_i,\ i=1,\ldots,d\}$}
for $\mbf{x},\mbf{y}\in\R^d$ ($(\mbf{x},\mbf{y}]=\emptyset,$ if $x_i\geq y_i$ for some $i\in\{1,\ldots,d\}$). A block in $\Z^d$ is the intersection of a block in $\R^d$ and the set $\Z^d.$ In particular, for a block $B=(\mbf{s},\mbf{t}]\subseteq [0,1]^d$ and $n\in\N$, we denote the associated block $n B \cap \Z^d=(\gaus{n\mbf{s}},\gaus{n\mbf{t}}]\cap\Z^d$ by $B_n$. Writing $\lambda$ for the Lebesgue measure on $\R^d$, it then holds that $\lambda((\gaus{n\mbf{s}},\gaus{n\mbf{t}}])=\#B_n$.

We say a block $W$ in $\Z^d$ belongs standardly to a block $U\subset\Z^d$ and denote this by $W \triangleleft U$ whenever $W\subset U$ and the minimal vertices of $W$ and $U$ (in the sense of the lexicographic order) coincide.
Let $a_j^{(i)},b_j^{(i)}$ ($i=1,\ldots,d$, $j=1,\ldots,n_i$, $n_i\in\N$) be real numbers with $0=a_1^{(i)}<b_1^{(i)}<a_2^{(i)}<b_2^{(i)}<\cdots<a_{n_i}^{(i)}<b_{n_i}^{(i)}=1$ for $i=1,\dots,d$. We say a collection of blocks is strongly separated (see \cite{deoFCLT1975}) if it is a subfamily of blocks of the form 
\[\left\{\prod_{i=1}^d (a_{k_i}^{(i)},b_{k_i}^{(i)}]:\, 1\leq k_i\leq n_i, 1\leq i\leq d\right\}.\]

Denoting the supremum norm on $\R^d$ by $\|\cdot\|_\infty$, we define the distance 
\[\text{dist}(S,Q)=\inf\{\|\mbf{x}-\mbf{y}\|_\infty:\, \mbf{x}\in S, \mbf{y}\in Q\}\]
between two sets $S$ and $Q$.
Given observations $(X_\mbf{j})_{\undnr{1}\leq\mbf{j}\leq\undnr{n}}$ ($n\in\N$), a real-valued random field $\{V_n(\mbf{i})\}_{\undnr{1}\leq\mbf{i}\leq\undnr{n}}$ will be called a dependent multiplier field with bandwidth $q=q_n$ if it is a Gaussian random field, independent of $(X_\mbf{j})_{\undnr{1}\leq\mbf{j}\leq\undnr{n}}$ , with $\mathrm{E} V_{n}(\mathbf{i})=0$, $\operatorname{var}\{V_{n}(\mathbf{i})\}=1$ and
\begin{equation*}
\operatorname{cov}\left(V_{n}(\mathbf{i}),V_{n}(\mathbf{j})\right)=\omega((\mathbf{i}-\mathbf{j})/q)
\end{equation*}
for a symmetric bounded function $\omega$ that is continuous at zero with $\omega(\undnr{0})=1$ and 
\[\sum\limits_{-\undnr{n}\leq\mbf{j}\leq \undnr{n}}|\omega(\mbf{j}/q)|=\mathcal{O}(q^d).\]

We consider a separable (real) Hilbert space $H$ with inner product $\sprod{\cdot,\cdot}$ and associated norm $\|x\| = \sqrt{|\sprod{x,x}|}$. (Since $\R^k$ with the inner product $\sprod{x,y}=x^\top y$ is also a Hilbert space, we will also denote the usual $l_2$-norm in $\R^k$ by $\|\cdot\|$.) Unless stated otherwise, the spaces considered are always seen as measurable spaces with their Borel $\sigma$-algebra. Let $L(H,H)$ be the space of bounded (with respect to the operator norm \\$\|S\|=\sup\{\|S(h)\|:\, h\in H, \|h\|\leq 1\}$) linear operators from $H$ to $H$. $\mathcal{S}(H)$ denotes the set of all self-adjoint positive nuclear operators in $L(H,H)$. The notation $\{e_k\}_{k\in\N}$ is used for complete orthonormal systems in $H$.  The trace of a nuclear operator $S\in\mathcal{S}(H)$ is $\rm tr(S)=\sum_{i=1}^\infty \sprod{Se_i,e_i}$, and $\|S-S'\|_{\rm tr}=\rm tr(S-S')$ defines a metric on $\mathcal{S}(H)$. Consider the span $H_k$ of the first $k$ $e_i$. Then the orthogonal projections on $H_k$ are $P_k: H\to H_k$, $h\mapsto \sum_{i=1}^k \sprod{h,e_i}e_i$, and the corresponding complementary operators are $A_k:H\to H$, $h\mapsto h-\sum_{i=1}^k \sprod{h,e_i}e_i=\sum_{i=k+1}^\infty \sprod{h,e_i}e_i$. For any $H$-valued random variable, we write $X^{(k)}=P_k X$ and $X^{k}=<X,e_k>$. 

For $\mbf{t}\in[0,1]^d$, let $R_i\in\{<,\geq\}$, $i=1,\dots,d$, and define a quadrant in $[0,1]^d$ by 
\[Q_{R_1,\ldots,R_d}(\mbf{t})=\{(s_1,\ldots,s_d)\in[0,1]^d:\, s_i R_i t_i, i=1,\ldots,d\}.\]
We say that a function $x:[0,1]^d\to H$ has quadrant limits if $x_{Q(\mbf{t})}=\lim_{\mbf{s}\to\mbf{t},\mbf{s}\in Q(\mbf{t})}x(\mbf{s})$ exists for all $\mbf{t}\in [0,1]^d$ and quadrants $Q(\mbf{t})$. $x$ is called continuous from above if $x(\mbf{t})=x_{Q_{\geq,\ldots,\geq}(\mbf{t})}(\mbf{t})$ for all $\mbf{t}\in [0,1]^d$.
In analogy to the case $H=\R$, we will consider stochastic processes in the space $D_H([0,1]^d)$ of functions $x:[0,1]^d\to H$ that have quadrant limits and are continuous from above. 
We endow $D_H([0,1]^d)$ with the metric
	\[d_S(x,y)=\inf\limits_{\lambda\in\Lambda}\{\max\{\sup\limits_{\mbf{t}\in[0,1]^d}\|x(\mbf{t})-y(\lambda(\mbf{t}))\|,\sup\limits_{\mbf{t}\in[0,1]^d}\|\mbf{t}-\lambda(\mbf{t})\|\}\}, \]
	where 
	\begin{align*}
	\Lambda&=\{\lambda:[0,1]^d\to [0,1]^d:\, \lambda(t_1,\ldots,t_d)=\left(\lambda_1(t_1),\ldots,\lambda_d(t_d)\right),\, \lambda_p:[0,1]\to[0,1] \\
	& \text{cont., strictly increasing and } \lambda_p(0)=0,\lambda_p(1)=1 \text{ for all }p=1,\ldots,d\}
	\end{align*}
	 (see, e.g., \cite{bickel1971,neuhaus1969,neuhaus1971} for $D_\R([0,1]^d)$). Let $C_H([0,1]^d)$ be the subset of functions in $D_H([0,1]^d)$ that are continuous with respect to the supremum-norm $\|x\|_\infty=\sup\{\|x(\mbf{t})\|:\, \mbf{t}\in[0,1]^d\}$.

It can be seen that the proofs which Neuhaus \cite{neuhaus1969} (see also \cite{neuhaus1971}) provides for $D_\R([0,1]^d)$ can be extended to our present setting with only minor changes. In particular, $(D_H([0,1]^d),d_S)$ is separable and (topologically) complete and the Borel $\sigma$-algebra coincides with the $\sigma$-algebra generated by the coordinate mappings (for dense subsets of $[0,1]^d$). The relation between $d_S$ and the supremum norm on $D_H([0,1]^d)$ is the same as in $D_\R([0,1]^d)$, and $(C_{H}([0,1]^d),\|\cdot\|_\infty)$ is a separable Banach space with $C_H([0,1]^d)\subseteq D_H([0,1]^d)$. 
If $(X_\mbf{t})_{\mbf{t}\in[0,1]^d}$ is a stochastic process with values in $D_H([0,1]^d)$, then the increment $X(B)$ of $X$ around a block $B=\prod_{i=1}^d(s_i,t_i]$ is given by
\[X(B)=\sum_{\varepsilon_1=0,1}\cdots \sum_{\varepsilon_d=0,1} (-1)^{d-\sum_{i=1}^d\varepsilon_i}X\left(s_1+\varepsilon_1(t_1-s_1),\ldots,s_d+\varepsilon_d(t_d-s_d)\right), \]
where we use the notations $X(\mbf{t})$ and $X_{\mbf{t}}$ synonymously.
For ease of notation, we will often write this as
\[X(B)=\sum\limits_{\boldsymbol{\varepsilon}\in\{0,1\}^d}(-1)^{d-\sum_{j=1}^d \varepsilon_j} X\left(\mbf{s}+\boldsymbol{\varepsilon}(\mbf{t}-\mbf{s})\right). \]
Since $X_\mbf{t}=X((\undnr{0},\mbf{t}])$ a.s. for a process which vanishes at zero (i.e., $X_\mbf{s}=0$ a.s. for any $\mbf{s}\in[0,1]^d$ with $\min s_i = 0$), we often denote $X((\undnr{0},\mbf{t}])$ and $X(n(\undnr{0},\mbf{t}])$ by $X(\mbf{t})$ and $X_n(\mbf{t})$ respectively. For $\mbf{k},\mbf{m}\in\Z^d$ and $\{x_\mbf{j}\}_{\mbf{j}\in\Z^d}$, we write 
\[\sum\limits_{\mbf{k}<\mbf{j}\leq \mbf{m}}x_{\mbf{j}}=\begin{cases}
\sum\limits_{\mbf{j}\in(\mbf{k},\mbf{m}]\cap\Z^d}x_{\mbf{j}}, & \mbf{k}<\mbf{m}\\
\sum\limits_{\mbf{j}\in\emptyset}x_{\mbf{j}}=0, & \mbf{k}\nless \mbf{m}.
\end{cases}\]
We will now define the Hilbert space valued analogue of the Brownian sheet (or Chentsov process):
\begin{defi}
An $H$-valued stochastic process $X=(X_\mbf{t})_{\mbf{t}\in[0,1]^d}$ is a Brownian sheet in $H$ with covariance operator $S\in\mathcal{S}(H)$ iff
	\begin{enumerate}
	\item $\Pr\left(X \in C_H([0,1]^d)\right)=1$,
	\item $X_\mbf{t}=0$ a.s. if $[\mbf{t}]=0$ and 
	\item for pairwise disjoint blocks $B_1,\ldots,B_m$ in $[0,1]^d$, the increments $X(B_1),\ldots,X(B_m)$ are independent Gaussian random elements in $H$ with mean zero and covariance operators $\lambda(B_i)S$, where $S\in\mathcal{S}(H)$ does not depend on $B_i$.
\end{enumerate}
\end{defi}

\begin{remark}
\begin{itemize}
	\item In order to see that the independence and Gaussian distribution of the increments over pairwise disjoint blocks yields a Gaussian process, one can proceed analogously to the one-dimensional case and write any linear combination of $X_{\mbf{t}_i}=X((\undnr{0},\mbf{t}_i])$ for points $\mbf{t}_i\in[0,1]^d$ ($i=1,\ldots,l$) as a linear combination of increments over pairwise disjoint blocks whose union is $\cup_{i=1}^l (\undnr{0},\mbf{t}_i]$. Note that the increments have the following additivity property: If a block $C$ can be partitioned into two disjoint blocks $A$ and $B$, then $X(C)=X(A\cup B)=X(A)+X(B)$ almost surely.
	\item If $X=(X_\mbf{t})_{\mbf{t}\in[0,1]^d}$ is a Brownian sheet in $H$ with covariance operator $S\in\mathcal{S}(H)$, then $(\sprod{X(\mbf{t}),h})_{\mbf{t}\in[0,1]^d}$ is a Brownian sheet with variance $\sprod{Sh,h}$ in $\R$ for any $h\in H$.
\end{itemize}
\end{remark}

For a $\sigma$-algebra $\mathcal{A}$, we define $L^p(\mathcal{A},H)$ as the set of all $\mathcal{A}$-measurable $H$-valued random elements $X$ with $\|X\|_p=\left(\mathrm{E}\|X\|^p\right)^{1/p}<\infty$.

As a measure of dependence, we use the following mixing conditions: For two $\sigma$-algebras $\mathcal{A}$ and $\mathcal{B}$, we can define the usual strong mixing coefficients
\[\alpha(\mathcal{A},\mathcal{B})=\sup\left\{|\Pr(A \cap B) - \Pr(A)\Pr(B)|:\; A \in\mathcal{A}, B\in\mathcal{B}\right\} \]
as well as the $\rho$-mixing coefficients
\[\rho_{\R}(\mathcal{A},\mathcal{B})=\sup\left\{\frac{|\Cov(X,Y)|}{\sqrt{\var(X) \var(Y)}}:\; X\in L^2(\mathcal{A},\R), Y \in L^2(\mathcal{B},\R), \var(X),\var(Y)>0 \right\}, \]
which lead to the following types of mixing coefficients for random fields. For a set $M\subset \Z^d$ let $\mathcal{A}_M=\sigma(X_\mbf{k}:\, \mbf{k}\in M)$ and define
\begin{equation}\label{eq:rhomixing}
\begin{aligned}
\rho_{\R}(r) = \sup\{\rho_{\R}(\mathcal{A}_M,\mathcal{A}_N):\;& M,N\subseteq \Z^d, \exists i\in\{1,\ldots,d\}\,\exists A,B\subset \Z, \text{dist}(A,B)\geq r:\\
&\,\forall \mbf{j}\in M,\mbf{k}\in N:\, j_i\in A, k_i\in B   \} 
\end{aligned}
\end{equation}
and
\[\rho_{\R}^*(r) = \sup\{\rho_{\R}(\mathcal{A}_M,\mathcal{A}_N):\; M,N\subseteq \Z^d, \text{dist}(M,N)\geq r\}. \]
As usual, we say that a random field is $\rho_{\R}$-mixing ($\rho_{\R}^*$-mixing), if $\lim_{r\to\infty}\rho_{\R}(r)=0$ ($\lim_{r\to\infty}\rho_{\R}^*(r)=0$).  

Finally, we use an $\alpha$-mixing coefficient where the cardinality of the index sets is restricted: For $k,m\in\N$, define
\[\alpha_{k,m}(r) = \sup\{\alpha(\mathcal{A}_M,\mathcal{A}_N):\; M,N\subseteq \Z^d,\, \text{dist}(M,N)\geq r,\, \#M\leq k, \#N\leq m\}. \]
To specify that the mixing coefficients belong to a process $X$, we use the notation $\rho_{\R,X}$ ($\rho_{\R,X}^*$ or $\alpha_{k,m,X}$).

For the two mixing coefficients, the inequality $\alpha(\mathcal{A},\mathcal{B})\leq \frac{1}{4}\rho_{\R}(\mathcal{A},\mathcal{B})$ holds. As in the case of time series, one could also define other mixing coefficients (see, e.g., \cite{deoFCLT1975} or \cite{LinLu} for $\varphi$-mixing). However, compared to the classical mixing conditions used for time series, the types of mixing considered for random fields are often stronger insofar as they allow interlaced sets in the suprema above.  While mixing conditions are not easy to verify in practice (see, e.g., \cite{doukhan1994}, \cite{guyon1995} for linear random fields), they are very common in the literature. For a thorough review of mixing conditions, see the monographs \cite{bradley2007} and \cite{doukhan1994}. An alternative measure of dependence for random fields has recently been proposed in \cite{el2013central}.
\end{subsection}

\end{section}

\begin{section}{Main results}

\begin{subsection}{Change-point problem for random fields}

We now present our FCLT for Hilbert space valued $\rho_\R$-mixing random fields. For real-valued $\rho$-mixing random fields, Kim and Seok \cite{kimSeok1995} used an approach proposed by Ibragimov \cite{ibragimov1975} to prove the FCLT under an additional assumption on the growth of the variance of the partial sums. Here, we have used a $\rho$-mixing condition that is stronger than the one in \cite{kimSeok1995} (we allow interlaced index sets in \eqref{eq:rhomixing}) and, since it is unclear how the growth condition would translate to the Hilbert space context, we use assumption 2 (see below) on the $\alpha$-mixing coefficients instead, which implies condition (2.6) in Corollary 2.3 of \cite{kimSeok1995}. However, although our assumptions are therefore stronger for real-valued fields, the following result is applicable not only to this special case but to general separable Hilbert spaces. As a byproduct of our proof, we extend a result from \cite{deoFCLT1975} to multivariate $\rho_\R$-mixing random fields. 

\begin{theorem}\label{thm:FCLT}
Let $\{X_\mbf{j}\}_{\mbf{j}\in\Z^d}$ be a strictly stationary $H$-valued random field with $\mathrm{E}X_{\undnr{1}}=\mu$. Assume that $\{X_\mbf{j}\}_{\mbf{j}\in\Z^d}$ is $\rho_{\R}$-mixing and that the following conditions hold for some $\delta>0$:
\begin{enumerate}
	\item $\mathrm{E}\|X_{\undnr{1}}\|^{2+\delta} < \infty$ 
	\item $\sum_{m\geq 1}m^{d-1}\alpha_{1,1}(m)^{\delta/(2+\delta)}<\infty$	
\end{enumerate} 
Then
\[
\left\{\frac{1}{n^{d/2}}\sum\limits_{\undnr{1}\leq \mbf{j}\leq \gaus{n\mbf{t}}} (X_\mbf{j}-\mu) \right\}_{\mbf{t}\in[0,1]^d} \Rightarrow \{W(\mbf{t})\}_{\mbf{t}\in[0,1]^d},
\]
where $\{W(\mbf{t})\}_{\mbf{t}\in[0,1]^d}$ is a Brownian sheet in $H$ and $W({\undnr{1}})$ has the covariance operator $S\in\mathcal{S}(H)$, defined by 
\begin{equation}
\sprod{Sx,y} = \sum_{\mbf{k}\in\Z^d}\mathrm{E}[\sprod{X_{\undnr{0}}-\mu,x}\sprod{X_\mbf{k}-\mu,y}],\quad \text{for }x,y\in H.
\label{eq:CovOp}
\end{equation}
Furthermore, the series in \eqref{eq:CovOp} converges absolutely.
\end{theorem}

\begin{remark}
Theorem \ref{thm:FCLT} could also be viewed as an extension of the central limit theorem proven by Tone \cite{tone2011central} for Hilbert space valued $\rho_\R$-mixing random fields to an FCLT. However, we do not build directly on Tone's result but present a proof which is more similar (see the proof of Lemma \ref{le:FCLTRkMixing}) to the proof of the FCLT in \cite{deoFCLT1975}. 
Compared to Tone's result, we employ both a stronger integrability assumption (1. in Theorem \ref{thm:FCLT}) and an additional assumption on the $\alpha$-mixing rate (2. in Theorem \ref{thm:FCLT}). Assumption 1. is used in order to obtain the maximal inequality \eqref{eq:maxIneq} in Lemma \ref{le:RosZhang}, which in turn allows us to infer both the tightness of the process (see Lemma \ref{le:FCLTRkMixing}) and point (c) of Lemma \ref{le:ChenWhite}. The $\alpha$-mixing assumption 2. yields the absolute convergence in \eqref{eq:CovOp}. This additional information is later used in the proof of Theorem \ref{theo2}. One could replace assumption 2. by any assumption that implies the absolute convergence of the covariance series. Alternatively, one could use the result in \cite{tone2011central} to obtain Theorem \ref{thm:FCLT} without the absolute convergence of \eqref{eq:CovOp} (see the proof of Lemma \ref{le:FCLTRkMixing}: point $(i)$ follows directly from Theorems 3.1 and 3.2 in \cite{tone2011central}).
\end{remark}

\noindent This can be used for the following change-point problem: Given observations $\{X_\mbf{j}\}_{\mbf{j}\in\{1,\ldots,n\}^d}$ with values in $H$, we want to test the null-hypothesis
\[H: \quad  \mathrm{E}X_\mbf{j}=\mu \quad \forall\; \mbf{k}\in\{1,\ldots,n\}^d \]
against the epidemic change alternative 
\begin{align*}
H_A: \quad &\exists\: \undnr{1} \leq \mbf{k}_0 <\mbf{m}_0 \leq \undnr{n}: \mathrm{E}X_\mbf{k} = \begin{cases}\mu, & \mbf{k}\in\{1,\ldots,n\}^d\setminus (\mbf{k}_0 ,\mbf{m}_0] \\ \mu+ \delta, &  \mbf{k}\in (\mbf{k}_0 ,\mbf{m}_0], \end{cases} 
\end{align*}
where $\mu,\delta\in H$ and $\mbf{k}_0,\mbf{m}_0$ are unknown. CUSUM-type asymptotic tests for the epidemic change in the mean problem have been investigated, e.g., in \cite{csorgoHorvath}, \cite{linJaruskova}, \cite{rackauskas} and \cite{yao1993} for real-valued time series. These were extended to i.i.d. random fields by Zemlys \cite{zemlysPHD} - who used an approach similar to \cite{rackauskas} - and Jaru\v{s}kov{\'a} and Piterbarg \cite{jaruskovaPiterbarg}. For weakly dependent random fields, \cite{bucchia2014} gave an extension of some results from \cite{jaruskovaPiterbarg}. The epidemic change problem for weakly dependent time series of functional observations was treated by Aston and Kirch \cite{astonKirch}, who constructed asymptotic tests based on the principal components of the data.

Consider the test statistic
\[T_{n}=\max\limits_{\undnr{0}\leq\mbf{k}<\mbf{m}\leq\undnr{n}}\frac{1}{n^{d/2}}\left\|\sum\limits_{\mbf{k}<\mbf{j}\leq\mbf{m}}X_\mbf{j} - \frac{[\mbf{m}-\mbf{k}]}{n^d}\sum\limits_{\undnr{1}\leq\mbf{j}\leq\undnr{n}}X_\mbf{j}\right\|. \]
Analogously to the univariate case, since both the maximum function and the Hilbert space norm are continuous, Theorem \ref{thm:FCLT} together with the continuous mapping theorem can be used to obtain the limit distribution of these statistics under $H$:
\begin{corollary}\label{cor:chpTest}
Under the assumptions of Theorem \ref{thm:FCLT}, it holds that
\[T_n \Rightarrow \sup\limits_{\undnr{0}\leq \mbf{s}<\mbf{t}\leq\undnr{1}}\left\|W(\mbf{s},\mbf{t}] - [\mbf{t}-\mbf{s}]W(\undnr{1})\right\|=T,\]
where $\{W(\mbf{t})\}_{\mbf{t}\in[0,1]^d}$ is the $H$-valued Brownian sheet defined in Theorem \ref{thm:FCLT}.
\end{corollary}

\begin{remark}
Here and in the following, we focus on the problem of detecting rectangular change-sets. For random fields, whose index set is a grid with rectangular mesh, rectangular sets or their unions are in a sense a natural fit. While it would in principle be possible to extend the testing procedure presented here to other classes of sets, doing so would pose additional technical challenges that are beyond the scope of this paper.
From a technical point of view, rectangular change-sets have the advantage that partial sums over such sets can be rewritten as sums and differences of partial sums over rectangles whose lower edge is zero. This makes it possible to exploit the rich theory of vector-indexed processes and in particular the well developed theory of weak convergence (see, e.g., \cite{bickel1971} and \cite{neuhaus1969}). Furthermore, the result in \cite{moricz1983} provides handy maximal inequalities in this setting, which are essential tools to prove the tightness of a partial sum process (see, e.g., the proofs of Lemma \ref{le:FCLTRkMixing} and Theorem \ref{thm:FCLT} below).
\end{remark}

For $\R^p$-valued observations $\{X_\mbf{j}\}_{\mbf{j}\in\{1,\ldots,n\}^d}$, this result can be used to obtain a test for the change in distribution problem of testing
\[H: \quad  F(\mbf{t})=\Pr(X_\mbf{i}\leq \mbf{t})\quad \forall\; \mbf{i}\in\{1,\ldots,n\}^d, \mbf{t}\in\R^p \]
against the alternative 
\begin{align*}
H_A: \quad &\exists\: \undnr{1} \leq \mbf{k}_0 <\mbf{m}_0 \leq \undnr{n}: \Pr(X_\mbf{k}\leq \mbf{t})= \begin{cases}F(\mbf{t}), & \mbf{k}\in\{1,\ldots,n\}^d\setminus (\mbf{k}_0 ,\mbf{m}_0] \\ G(\mbf{t}), &  \mbf{k}\in (\mbf{k}_0 ,\mbf{m}_0], \end{cases} 
\end{align*}
where the distribution functions $F$ and $G$, $F\neq G$, are unknown. Our goal is to write this as a change in mean problem for a suited Hilbert space. Common test statistics depend on the empirical distribution functions as estimators for the unknown parameters $F$ and $G$. These are sums over the indicator functions $\mathbbm{1}_{\{X_\mbf{j}\leq\mbf{t}\}}$, $\mbf{t}\in\R^p$. For some nonnegative, bounded weight function $w:\R^p\to\R$ with $\int_{\R^p}w(\mbf{t})d\mbf{t}<\infty$, the latter can be interpreted as random elements of the Hilbert space $L^2(\R^p,w)$ of measurable functions $f:\R^p \to \R$, with $\|f\|<\infty$ for the norm induced by the inner product
\[\sprod{f,g} = \int_{\R^p}f(\mbf{t})g(\mbf{t})w(\mbf{t})d\mbf{t}.\]
It is common to choose a density as the weight function $w$, since this leads to a scalar product of the form 
\[\sprod{f,g} = \int_{\R^p}f(\mbf{t})g(\mbf{t})dF,\]
where $F$ is the corresponding distribution function, and a Cram{\'e}r-von Mises type test. When no information on the true distribution of the data is available, an obvious choice is the density of the normal distribution. However, any function that fulfills the above requirements can in principle be used.
If $F$ is the distribution function of $X_\mbf{j}$, it can be seen that for any $h \in L^2(\R^p,w)$,
\[\mathrm{E}\left[\sprod{\mathbbm{1}_{\{X_\mbf{j}\leq\cdot\}},h}\right] = \mathrm{E}\left[\int_{\R^p}\mathbbm{1}_{\{X_\mbf{j}\leq\mbf{t}\}}h(\mbf{t})w(\mbf{t})d\mbf{t}\right] = \int_{\R^p} F(\mbf{t})h(\mbf{t})w(\mbf{t})d\mbf{t} = \sprod{F,h} \]
by Fubini's theorem. Therefore, $F$ is the expected value of $\mathbbm{1}_{\{X_\mbf{j}\leq\cdot\}}$ in $L^2(\R^p,w)$ and we obtain a Cram\'er-von Mises type test for the change in distribution problem by translating Corollary \ref{cor:chpTest} for this special case:
\begin{corollary} \label{cor:chDistr}
Let $\{X_\mbf{j}\}_{\mbf{j}\in\Z^d}$ be an $\R^p$-valued stationary random field with marginal distribution function $F$, which is $\rho_{\R}$-mixing with $\alpha$-mixing coefficients that satisfy 
\[\sum_{m\geq 1}m^{d-1}\alpha_{1,1}(m)^{\delta/(2+\delta)}<\infty\]
for some $\delta>0$. The change-point statistic
\[T_{n,w}=\max\limits_{\undnr{0}\leq\mbf{k}<\mbf{m}\leq\undnr{n}}\frac{1}{n^d} \int\limits_{\R^p}\left(\sum\limits_{\mbf{k}<\mbf{j}\leq\mbf{m}}\mathbbm{1}_{\{X_\mbf{j}\leq\mbf{x}\}} - \frac{[\mbf{m}-\mbf{k}]}{n^d}\sum\limits_{\undnr{1}\leq\mbf{j}\leq\undnr{n}}\mathbbm{1}_{\{X_\mbf{j}\leq \mbf{x}\}} \right)^2 w(\mbf{x})d\mbf{x}\]
then satisfies
\[T_{n,w} \Rightarrow \sup\limits_{\undnr{0}\leq \mbf{s}<\mbf{t}\leq\undnr{1}}\left\|W(\mbf{s},\mbf{t}] - [\mbf{t}-\mbf{s}]W(\undnr{1})\right\|^2 = T_w,\]
where $\{W(\mbf{t})\}_{\mbf{t}\in[0,1]^d}$ is a Brownian sheet in $L^2(\R^p,w)$ and $W(\undnr{1})$ has the covariance operator $S\in\mathcal{S}\left(L^2(\R^p,w)\right)$ defined by 
\begin{equation*}
\sprod{Sx,y} = \sum_{\mbf{k}\in\Z^d}\mathrm{E}\left[\int_{\R^p}\left(\mathbbm{1}_{\{X_{\undnr{0}}\leq\mbf{t}\}}- F(\mbf{t})\right)x(\mbf{t}) w(\mbf{t})d\mbf{t} \int_{\R^p}\left(\mathbbm{1}_{\{X_{\mbf{k}}\leq\mbf{t}\}}- F(\mbf{t})\right)y(\mbf{t}) w(\mbf{t})d\mbf{t}\right],
\end{equation*}
for $x,y\in L^2(\R^p,w)$.
\end{corollary}
Note that since $x \mapsto \mathbbm{1}_{\{x\leq \cdot\}}$ is a measurable bijection, the mixing properties of $\{X_\mbf{j}\}_{\mbf{j}\in\Z^d}$ are preserved. Due to the non-negativity and integrability of $w$, the moment condition of Theorem \ref{thm:FCLT} is satisfied.
\end{subsection}

\begin{subsection}{Dependent wild bootstrap for change-point detection}
We formulate our theorem on the consistency of the bootstrap version of the partial sum process for Hilbert space valued random fields. 

\begin{theorem}\label{theo2} Let the assumptions of Theorem \ref{thm:FCLT} hold and assume additionally that 
\begin{equation}
\sum_{m\geq 1}m^{d-1}\alpha_{2,2}(m)^{\delta/(2+\delta)}<\infty
\label{eq:assAlpha2}
\end{equation} 
and $\mathrm{E}\|X_{\undnr{1}}\|^{4+2\delta} <\infty$. Furthermore, let $\{V_{n,1}(\mathbf{i})\}_{\undnr{1}\leq\mathbf{i}\leq\undnr{n}},\ldots, \{V_{n,K}(\mathbf{i})\}_{\undnr{1}\leq\mathbf{i}\leq\undnr{n}}$ ($K\in\N$) be independent copies of the same dependent multiplier field. \footnote{Note that the restriction on $\omega$ is weaker than the restriction $\omega(\mbf{j}/q)=0$ for $\mbf{j}$ with $\max_i|j_i| \geq q$ used in \cite{bHeuser}, but since the proofs remain essentially unaffected, all results from that paper are still applicable.}
Lastly, let the bandwidth $q=q_n$ fulfill $q_n\rightarrow\infty$ and $q_n=o(\sqrt{n})$. Then
\begin{equation*}
\left(S_n,S_{n,1}^\star,\ldots,S_{n,K}^\star\right)\Rightarrow\left(W,W_1^\star,\ldots,W_K^\star \right) \quad\text{in $D_H([0,1]^d)^{K+1}$}
\end{equation*}
where $S_n$ is the partial sum process \eqref{eq:partialSum}, $S_{n,1}^\star,\ldots,S_{n,K}^\star$ are bootstrapped partial sum processes defined as in \eqref{eq:bootstrapDef} and $W_1^\star,\ldots,W_K^\star$ are independent copies of the Hilbert space valued Brownian sheet $W$ from Theorem \ref{thm:FCLT}. 
\end{theorem}

The additional assumption \eqref{eq:assAlpha2} and the stronger integrability assumption are used to obtain the convergence of the implicit long-run variance estimators (see Lemma \ref{rem:LRV} below). Alternatively, in order to use the proof presented, for instance, in \cite{lavancier2008} for bandwidths $q=o(n)$, one could replace assumption \eqref{eq:assAlpha2} by even stronger mixing and integrability conditions (see, e.g., \cite{guyon1995}, Lemma 4.6.2) in order to obtain the summability of the fourth-order cumulants (see Assumption (Y2) in \cite{bHeuser}). 

Write $T_{n,1}^\star,\ldots,T_{n,K}^\star$ and $T_{n,w,1}^\star,\ldots,T_{n,w,K}^\star$ for the bootstrapped analogues of the above change-point statistics $T_n$ and $T_{n,w}$, where $X_\mbf{j}$ and $\mathbbm{1}_{\{X_\mbf{j}\leq \cdot\}}$ are replaced by 
$V_{n,l}(\mathbf{j})\left\{X_{\mathbf{j}}-\hat{\mu}(\mbf{j})\right\}$ and $V_{n,l}(\mathbf{j})\left\{\mathbbm{1}_{\{X_\mbf{j}\leq \cdot\}}-\hat{\mu}(\mbf{j})\right\}$, respectively ($l=1,\ldots,K$). As a direct consequence of Theorem \ref{theo2}, we obtain the same limit distributions as for the original statistics:
\begin{corollary}
\begin{itemize}
	\item[(a)] Let the assumptions of Theorem \ref{theo2} hold. Then it holds that
	\[(T_n,T_{n,1}^\star,\ldots,T_{n,K}^\star) \Rightarrow (T,T_1^\star,\ldots,T_K^\star),\]
where $T_1^\star,\ldots,T_K^\star$ are independent copies of $T$.	
	\item[(b)] Let $\{X_\mbf{j}\}_{\mbf{j}\in\Z^d}$ be an $\R^p$-valued stationary random field that fulfills the assumptions of Corollary \ref{cor:chDistr} and \eqref{eq:assAlpha2}. Let $\{V_{n,1}(\mbf{j})\}_{\undnr{1}\leq\mbf{j}\leq\undnr{n}}$,\ldots,$\{V_{n,K}(\mbf{j})\}_{\undnr{1}\leq\mbf{j}\leq\undnr{n}}$ be as in Theorem \ref{theo2}. Then it holds that
\[(T_{n,w},T_{n,w,1}^\star,\ldots,T_{n,w,K}^\star) \Rightarrow (T_w,T_{w,1}^\star,\ldots, T_{w,K}^\star),\]
where $T_{w,1}^\star,\ldots,T_{w,K}^\star$ are independent copies of $T_w$.	
\end{itemize}
\end{corollary}
Using this corollary, we can obtain critical values for the test statistic $T_n$ (and analogously for $T_{n,w}$) in the following way: Simulate the $K$ conditionally independent copies $T^\star_{n,1},\ldots,T^\star_{n,K}$. For a given significance level $\alpha\in(0,1)$, calculate the $(1-\alpha)$ sample quantile $q^\star_{n,K}(1-\alpha)$ of $T^\star_{n,1},\ldots,T^\star_{n,K}$ and reject the hypothesis of stationarity if $T_n\geqq^\star_{n,K}(1-\alpha)$. Then Lemma F.1 in \cite{supplBuecher} yields $\lim_{K\to\infty}\lim_{n\to\infty}\Pr\left(T_n\geq q^\star_{n,K}(1-\alpha)\right)=\alpha$.
\end{subsection}
\end{section}

\begin{section}{Simulation study}
To illustrate the finite sample behavior of the Cram\'er-von Mises type change-point test (using $T_{n,w}$) with dependent wild bootstrap, we present the results of a small simulation study. We use the density of the $\mathcal{N}(100,1000^2)$-distribution as a weight function $w$ to define the Hilbert space $L^2(\R,w)$. As a data generating process, we use an autoregressive process
\[Y_\mbf{k} =  a Y_{k_1-1,k_2}+a Y_{k_1,k_2-1} - a^2 Y_{k_1-1,k_2-1} + \epsilon_{k_1,k_2}, \quad \mbf{k}\in\{1,\ldots,n\}^2 \]
for dimension $d=2$, where the parameter $a$, which reflects the dependence structure of the process, takes the values $a=0.2, 0.5$ and the innovations $\{\epsilon_\mbf{k}\}_{\mbf{k}\in\Z^d}$ are i.i.d. \mbox{$\mathcal{N}(0,(1-a^2)^d)$}-distributed. Applying the results in \cite{doukhan1994}, Section 2.1.1, it can be seen that this process fulfills the mixing assumptions of Theorems \ref{thm:FCLT} and \ref{theo2}. We use sample sizes $n=30, 40,50$.
We consider two types of changes in distribution, changes in the mean and changes in the skewness of the process, each over a change-set of the form $C=(\btheta,\bgamma]$ ($\undnr{0}<\btheta<\bgamma<\undnr{1}$). For the change in mean, we consider	
\[X^{(1)}_\mbf{k}=Y_\mbf{k}+\Delta \mathbbm{1}_{C_n}(\mbf{k}),\quad \mbf{k}\in \{1,\ldots,n\}^d,\]
with $\Delta=0,0.5,1$, $C_n=(\gaus{n\btheta},\gaus{n\bgamma}]\cap\Z^d$. For the change in skewness, we use the same approach as in \cite{sharipov2014} and simulate a second data generating process $\{Y'_\mbf{k}\}_{\mbf{k}\in\{1,\ldots,n\}^d}$ which is independent of $\{Y_\mbf{k}\}_{\mbf{k}\in\{1,\ldots,n\}^d}$, using the same scheme as for $\{Y_\mbf{k}\}_{\mbf{k}\in\{1,\ldots,n\}^d}$. We define 
\[X^{(2)}_\mbf{k}=\begin{cases}Y^2_\mbf{k}+Y^{'2}_\mbf{k}, & \mbf{k}\notin C_n \\ 4- (Y^2_\mbf{k}+Y^{'2}_\mbf{k}), & \mbf{k}\in C_n. \end{cases} \]
In order to investigate the effect of the volume (Vol) of the change-set on the power, we consider three different examples, where $C^{(1)}=(\btheta_1,\bgamma_1]$ is a small block, $C^{(2)}=(\btheta_2,\bgamma_2]$ is medium-sized and $C^{(3)}=(\btheta_3,\bgamma_3]$ is large (see Table \ref{tab:examplesShort}).
\begin{table*}
\footnotesize
\centering
\begin{tabular}{cccccc}
   Example 1 &  & Example 2  &  & Example 3 \\ \cmidrule{1-1} \cmidrule{3-3} \cmidrule{5-5}
  $ C^{(1)}=\begin{matrix}\left(\left(\begin{matrix}0.2\\0.3\end{matrix}\right),
    \left(\begin{matrix}0.6\\0.55\end{matrix}\right)\right]\\
    \textrm{Vol$=0.1$}\end{matrix}$
     &&$ C^{(2)}=\begin{matrix} \left(\left(\begin{matrix}0.1\\0.1\end{matrix}\right),
    \left(\begin{matrix}0.9\\0.85\end{matrix}\right)\right] \\ \textrm{Vol$=0.6$} \end{matrix}$ &&  $ C^{(3)}=\begin{matrix} \left(\left(\begin{matrix}0.05\\0.1\end{matrix}\right),
    \left(\begin{matrix}0.95\\1.0\end{matrix}\right)\right]\\ \textrm{Vol$=0.81$} \end{matrix}$\\
  \bottomrule
\end{tabular}
	\caption{Change sets $C^{(1)}$, $C^{(2)}$, $C^{(3)}$ with $C^{(i)}=(\btheta_i,\bgamma_i]$ for $i=1,2,3$ and corresponding volumes for the different examples.}\label{tab:examplesShort}
	\end{table*}
We compare two bootstrap methods:
	\begin{itemize}
		\item Discretely sampled Ornstein-Uhlenbeck sheets (autoregressive wild bootstrap (\textbf{AR}))
	\[V_n(\mbf{k}) = a V_n(k_1-1,k_2)+a V_n(k_1,k_2-1) - a^2 V_n(k_1-1,k_2-1) + \varepsilon_{n,k_1,k_2}, \]
with $a=\exp\left(-1/q(n)\right)$ and i.i.d. $\mathcal{N}\left(0,(1-a^2)^{d}\right)$-distributed innovations $\varepsilon_{n,\mbf{k}}$. This corresponds to the exponential kernel function $\omega_{q(n),\mbf{j}}=\prod_{i=1}^d\exp\left(-\frac{|j_i|}{q(n)}\right)$. 	
			\item Moving average random fields (\textbf{MA}): Let $\{\varepsilon_{\mbf{j}}\}_{\mbf{j}\in\Z^d}$ be a random field of i.i.d. $\mathcal{N}(0,1)$-distributed r.v. For $a=(q(n)+1)^{-d/2}$ (i.e., $a=|B_{q(n)/2}|^{-1/2}$, with $B_{q(n)/2}:=\left\{-q(n)/2,\ldots,q(n)/2\right\}^d$), we consider the process defined by
\[V_n(\mbf{k})=a \sum\limits_{\mbf{j}\in B_{q(n)/2}}\varepsilon_{\mbf{k}-\mbf{j}}. \]
			This corresponds to the Bartlett-type kernel function $\omega_{q(n),\mbf{j}}=\prod_{i=1}^d \left(1 - \frac{|j_i|}{q(n)+1}\right)^+$.		
	\end{itemize}
	For both methods, we consider $q=2, 6, 10$ to cover a wide range of possible bandwidths. We use the mean estimators $\hat{\mu}=F_n$ and 
	\[\hat{\mu}(\mbf{k})=\tilde{F}_n(\mbf{k})=\begin{cases}
	\frac{1}{\#\hat{C}_n}\sum_{\mathbf{i}\in\hat{C}_n}\mathbbm{1}_{\{X_{\mathbf{i}}\leq\cdot\}}\ \ \ &\text{if }\mathbf{k}\in\hat{C}_n,\\
\frac{1}{\#\hat{C}_n^c}\sum_{\mathbf{i}\notin\hat{C}_n}\mathbbm{1}_{\{X_{\mathbf{i}}\leq\cdot\}}\ \ \ &\text{if }\mathbf{k}\notin\hat{C}_n
\end{cases} \]
(see section 1.2). The change-set estimator $\hat{C}_n=(\hat{\mbf{k}},\hat{\mbf{m}}]$ used for $\tilde{F}_n$ is obtained by taking the maximizing values for the test statistic $T_{n,w}$ as estimators $\hat{\mbf{k}}$ and $\hat{\mbf{m}}$. This approach is very common in the literature, see, e.g., \cite{csorgoHorvath}, p.50. Alternatives would be the maximum likelihood or the least squares estimator.

The empirical size and power of the tests are estimated using $N=500$ repetitions, for each of which $J=500$ wild bootstrap-iterations are used to derive the critical values. The nominal size was chosen as $\alpha=0.05, 0.1$. Table \ref{tableHypothesisDim2} shows the empirical size of the tests. Unsurprisingly, for both choices of $a$ the empirical size depends strongly on the bandwidth $q$, which is a measure of the dependence of the bootstrap process. The greater $q$, the greater the dependence in the bootstrap sample and the smaller the empirical size of the test. For $\hat{\mu}=F_n$ and $a=0.2$, the nominal size is always held for $q=10$ and can be adequately held for $q=6$, whereas the  empirical size for $a=0.5$ tends to be greater than the nominal one even for $q=10$. For $\hat{\mu}=\tilde{F}_n$, the empirical size is much larger than the nominal one for all choices of $a$ and $q$. The over-rejection under the null hypothesis seems to be typical for bootstrap methods (see \cite{doukhan2015}). Conversely, under the alternative, the empirical power decreases with rising bandwidth $q$, but the effect is more pronounced for $\hat{\mu}=F_n$ than for $\hat{\mu}=\tilde{F}_n$ (see, e.g., Tables \ref{table2DimBootstrap_a0.2.txt_Delta0.5} and \ref{table2DimBootstrapWChEst_a0.2.txt_Delta0.5}). It is thus difficult to give a recommendation regarding the bandwidth.

This effect is however less important than the choice of change-set for the power of the test: Where both the change in mean and the change in skewness are well detected for medium-sized and large change-sets (Examples 2 and 3), the empirical power for small change-sets (Example 1) can be very small for $q=6,10$ (see Tables \ref{table2DimBootstrap_a0.2.txt_Delta0.5}, \ref{table2DimBootstrap_a0.5.txt_Delta0.5}, \ref{table2DimBootstrap_a0.2.txt_Delta1}, \ref{table2DimBootstrap_a0.5.txt_Delta1}, \ref{table2DimBootstrapSkewness_a0.2.txt} and \ref{table2DimBootstrapSkewness_a0.5.txt}). Again, the tests based on $\tilde{F}_n$ have a higher empirical power than the tests based on $F_n$ and retain their good detection properties even for small change-sets (see Tables \ref{table2DimBootstrapWChEst_a0.2.txt_Delta0.5}, \ref{table2DimBootstrapWChEst_a0.5.txt_Delta0.5}, \ref{table2DimBootstrapWChEst_a0.2.txt_Delta1}, \ref{table2DimBootstrapWChEst_a0.5.txt_Delta1}, \ref{table2DimBootstrapSkewnessChEst_a0.2.txt} and \ref{table2DimBootstrapSkewnessChEst_a0.5.txt}).
The tests perform better under weaker dependence in the observations, but for medium-sized and large change-sets the empirical power is good for both choices of $a$ and $\Delta=0.5$ and excellent for $\Delta=1$. Except for small change-sets and $\hat{\mu}=F_n$ (see, e.g., Table \ref{table2DimBootstrapSkewness_a0.5.txt}), the change in skewness is well detected by all procedures (see Tables \ref{table2DimBootstrapSkewness_a0.2.txt}-\ref{table2DimBootstrapSkewnessChEst_a0.5.txt}). Rising numbers $n$ of observations improve the empirical power of the tests.
The different choices of the random variables $\{V_n(\mathbf{i})\}_{\undnr{1}\leq\mathbf{i}\leq\undnr{n}}$ (\textbf{AR} or \textbf{MA}) do not seem to influence the power of the test strongly, with only slightly better empirical power under \textbf{MA} for $\hat{\mu}=F_n$ (see, e.g., Tables \ref{table2DimBootstrap_a0.5.txt_Delta0.5} and \ref{table2DimBootstrap_a0.2.txt_Delta1}).

\begin{table}[ht]
\centering
\caption{Hypothesis (stationarity)} 
\label{tableHypothesisDim2}
{\footnotesize
\resizebox{\columnwidth}{!}{%
\begin{tabular}{llll|rrc;{2pt/2pt}rrc;{2pt/2pt}rrc}
  \toprule 
 \multicolumn{4}{c}{} & \multicolumn{3}{c}{$n=30$} & \multicolumn{3}{c}{$n=40$} & \multicolumn{3}{c}{$n=50$}\\
 \cmidrule(l){5-7} \cmidrule(l){8-10} \cmidrule(l){11-13}
 & & & & q=2 & q=6 & q=10 & q=2 & q=6 & q=10 & q=2 & q=6 & q=10\\
 \midrule 
 \midrule\multirow{8}{*}{\textbf{$\hat{\mu}=F_n$}} & \multirow{4}{*}{\textbf{AR}} & \multirow{2}{*}{\textbf{$\alpha=0.05$}} & $\mathbf{a=0.2}$ & 0.18 & 0.01 & 0.00 & 0.19 & 0.04 & 0.00 & 0.17 & 0.04 & 0.01 \\ 
   &  &  & $\mathbf{a=0.5}$ & 0.58 & 0.08 & 0.00 & 0.65 & 0.14 & 0.03 & 0.67 & 0.15 & 0.04 \\ 
    &  & \multirow{2}{*}{\textbf{$\alpha=0.1$}} & $\mathbf{a=0.2}$ & 0.28 & 0.08 & 0.02 & 0.30 & 0.13 & 0.06 & 0.28 & 0.12 & 0.07 \\ 
   &  &  & $\mathbf{a=0.5}$ & 0.71 & 0.24 & 0.07 & 0.76 & 0.30 & 0.13 & 0.76 & 0.31 & 0.16 \\ 
   \cdashline{2-13} & \multirow{4}{*}{\textbf{MA}} & \multirow{2}{*}{\textbf{$\alpha=0.05$}} & $\mathbf{a=0.2}$ & 0.15 & 0.03 & 0.01 & 0.17 & 0.07 & 0.03 & 0.15 & 0.05 & 0.02 \\ 
   &  &  & $\mathbf{a=0.5}$ & 0.58 & 0.15 & 0.03 & 0.63 & 0.20 & 0.07 & 0.66 & 0.19 & 0.06 \\ 
   &  & \multirow{2}{*}{\textbf{$\alpha=0.1$}} & $\mathbf{a=0.2}$ & 0.28 & 0.12 & 0.03 & 0.28 & 0.15 & 0.09 & 0.26 & 0.13 & 0.09 \\ 
   &  &  & $\mathbf{a=0.5}$ & 0.71 & 0.29 & 0.13 & 0.76 & 0.34 & 0.18 & 0.77 & 0.33 & 0.19 \\ 
  \midrule\multirow{8}{*}{\textbf{$\hat{\mu}=\tilde{F}_n$}} & \multirow{4}{*}{\textbf{AR}} & \multirow{2}{*}{\textbf{$\alpha=0.05$}} & $\mathbf{a=0.2}$ & 0.26 & 0.18 & 0.13 & 0.24 & 0.17 & 0.14 & 0.20 & 0.13 & 0.13 \\ 
   &  &  & $\mathbf{a=0.5}$ & 0.68 & 0.40 & 0.31 & 0.71 & 0.38 & 0.29 & 0.71 & 0.35 & 0.26 \\ 
   &  & \multirow{2}{*}{\textbf{$\alpha=0.1$}} & $\mathbf{a=0.2}$ & 0.36 & 0.29 & 0.27 & 0.35 & 0.28 & 0.25 & 0.34 & 0.22 & 0.21 \\ 
   &  &  & $\mathbf{a=0.5}$ & 0.80 & 0.54 & 0.47 & 0.82 & 0.53 & 0.46 & 0.81 & 0.49 & 0.42 \\ 
   \cdashline{2-13} & \multirow{4}{*}{\textbf{MA}} & \multirow{2}{*}{\textbf{$\alpha=0.05$}} & $\mathbf{a=0.2}$ & 0.23 & 0.18 & 0.13 & 0.21 & 0.16 & 0.14 & 0.18 & 0.12 & 0.11 \\ 
   &  &  & $\mathbf{a=0.5}$ & 0.66 & 0.40 & 0.30 & 0.71 & 0.38 & 0.26 & 0.70 & 0.32 & 0.24 \\ 
   &  & \multirow{2}{*}{\textbf{$\alpha=0.1$}} & $\mathbf{a=0.2}$ & 0.32 & 0.26 & 0.25 & 0.33 & 0.24 & 0.23 & 0.29 & 0.19 & 0.18 \\ 
   &  &  & $\mathbf{a=0.5}$ & 0.77 & 0.51 & 0.44 & 0.79 & 0.49 & 0.41 & 0.80 & 0.46 & 0.38 \\ 
   \bottomrule 
\end{tabular}}
}
\end{table}

\begin{table}[ht]
\centering
\caption{Change in Mean, $\hat{\mu}=F_n$, a=0.2, $\Delta=0.5$} 
\label{table2DimBootstrap_a0.2.txt_Delta0.5}
{\footnotesize
\begin{tabular}{lllrrr|rrr|rrr}
  \toprule 
 \multicolumn{3}{c}{} & \multicolumn{3}{c}{$n=30$} & \multicolumn{3}{c}{$n=40$} & \multicolumn{3}{c}{$n=50$}\\
 \cmidrule(l){4-6} \cmidrule(l){7-9} \cmidrule(l){10-12}
 & & & q=2 & q=6 & q=10 & q=2 & q=6 & q=10 & q=2 & q=6 & q=10\\
 \midrule 
 \midrule\multirow{6}{*}{\textbf{AR}} & \multirow{3}{*}{\textbf{$\alpha=0.05$}} & \textbf{Ex. 1} & 0.43 & 0.03 & 0.00 & 0.66 & 0.20 & 0.03 & 0.83 & 0.34 & 0.05 \\ 
   &  & \textbf{Ex. 2} & 1.00 & 0.92 & 0.37 & 1.00 & 1.00 & 0.97 & 1.00 & 1.00 & 1.00 \\ 
    &  & \textbf{Ex. 3} & 0.81 & 0.38 & 0.07 & 0.97 & 0.85 & 0.57 & 1.00 & 0.99 & 0.93 \\ 
   \cdashline{2-12} & \multirow{3}{*}{\textbf{$\alpha=0.1$}} & \textbf{Ex. 1} & 0.57 & 0.21 & 0.04 & 0.78 & 0.43 & 0.19 & 0.91 & 0.61 & 0.30 \\ 
   &  & \textbf{Ex. 2} & 1.00 & 0.99 & 0.91 & 1.00 & 1.00 & 1.00 & 1.00 & 1.00 & 1.00 \\ 
   &  & \textbf{Ex. 3} & 0.90 & 0.66 & 0.42 & 0.99 & 0.94 & 0.87 & 1.00 & 1.00 & 0.99 \\ 
  \midrule\multirow{6}{*}{\textbf{MA}} & \multirow{3}{*}{\textbf{$\alpha=0.05$}} & \textbf{Ex. 1} & 0.43 & 0.12 & 0.01 & 0.65 & 0.33 & 0.10 & 0.85 & 0.50 & 0.18 \\ 
   &  & \textbf{Ex. 2} & 0.99 & 0.97 & 0.82 & 1.00 & 1.00 & 1.00 & 1.00 & 1.00 & 1.00 \\ 
   &  & \textbf{Ex. 3} & 0.81 & 0.54 & 0.27 & 0.96 & 0.88 & 0.80 & 1.00 & 1.00 & 0.98 \\ 
   \cdashline{2-12} & \multirow{3}{*}{\textbf{$\alpha=0.1$}} & \textbf{Ex. 1} & 0.56 & 0.28 & 0.10 & 0.79 & 0.51 & 0.29 & 0.91 & 0.72 & 0.45 \\ 
   &  & \textbf{Ex. 2} & 1.00 & 1.00 & 0.96 & 1.00 & 1.00 & 1.00 & 1.00 & 1.00 & 1.00 \\ 
   &  & \textbf{Ex. 3} & 0.88 & 0.74 & 0.56 & 0.99 & 0.94 & 0.90 & 1.00 & 1.00 & 1.00 \\ 
   \bottomrule 
\end{tabular}
}
\end{table}

\begin{table}[ht]
\centering
\caption{Change in Mean, $\hat{\mu}=\tilde{F}_n$, a=0.2, $\Delta=0.5$} 
\label{table2DimBootstrapWChEst_a0.2.txt_Delta0.5}
{\footnotesize
\begin{tabular}{lllrrr|rrr|rrr}
  \toprule 
 \multicolumn{3}{c}{} & \multicolumn{3}{c}{$n=30$} & \multicolumn{3}{c}{$n=40$} & \multicolumn{3}{c}{$n=50$}\\
 \cmidrule(l){4-6} \cmidrule(l){7-9} \cmidrule(l){10-12}
 & & & q=2 & q=6 & q=10 & q=2 & q=6 & q=10 & q=2 & q=6 & q=10\\
 \midrule 
 \midrule\multirow{6}{*}{\textbf{AR}} & \multirow{3}{*}{\textbf{$\alpha=0.05$}} & \textbf{Ex. 1} & 0.56 & 0.38 & 0.32 & 0.75 & 0.59 & 0.52 & 0.89 & 0.73 & 0.64 \\ 
   &  & \textbf{Ex. 2} & 1.00 & 0.99 & 0.98 & 1.00 & 1.00 & 1.00 & 1.00 & 1.00 & 1.00 \\ 
    &  & \textbf{Ex. 3} & 0.88 & 0.75 & 0.66 & 0.98 & 0.94 & 0.91 & 1.00 & 1.00 & 0.99 \\ 
   \cdashline{2-12} & \multirow{3}{*}{\textbf{$\alpha=0.1$}} & \textbf{Ex. 1} & 0.70 & 0.54 & 0.50 & 0.85 & 0.73 & 0.70 & 0.94 & 0.83 & 0.79 \\ 
   &  & \textbf{Ex. 2} & 1.00 & 1.00 & 1.00 & 1.00 & 1.00 & 1.00 & 1.00 & 1.00 & 1.00 \\ 
   &  & \textbf{Ex. 3} & 0.93 & 0.87 & 0.82 & 0.99 & 0.97 & 0.96 & 1.00 & 1.00 & 1.00 \\ 
  \midrule\multirow{6}{*}{\textbf{MA}} & \multirow{3}{*}{\textbf{$\alpha=0.05$}} & \textbf{Ex. 1} & 0.52 & 0.40 & 0.32 & 0.73 & 0.57 & 0.52 & 0.87 & 0.75 & 0.65 \\ 
   &  & \textbf{Ex. 2} & 0.99 & 0.99 & 0.98 & 1.00 & 1.00 & 1.00 & 1.00 & 1.00 & 1.00 \\ 
   &  & \textbf{Ex. 3} & 0.86 & 0.77 & 0.69 & 0.97 & 0.94 & 0.91 & 1.00 & 1.00 & 1.00 \\ 
   \cdashline{2-12} & \multirow{3}{*}{\textbf{$\alpha=0.1$}} & \textbf{Ex. 1} & 0.67 & 0.53 & 0.47 & 0.82 & 0.71 & 0.66 & 0.93 & 0.83 & 0.79 \\ 
   &  & \textbf{Ex. 2} & 1.00 & 1.00 & 1.00 & 1.00 & 1.00 & 1.00 & 1.00 & 1.00 & 1.00 \\ 
   &  & \textbf{Ex. 3} & 0.92 & 0.86 & 0.82 & 0.99 & 0.98 & 0.96 & 1.00 & 1.00 & 1.00 \\ 
   \bottomrule 
\end{tabular}
}
\end{table}

\begin{table}[ht]
\centering
\caption{Change in Mean, $\hat{\mu}=F_n$, a=0.5, $\Delta=0.5$} 
\label{table2DimBootstrap_a0.5.txt_Delta0.5}
{\footnotesize
\begin{tabular}{lllrrr|rrr|rrr}
  \toprule 
 \multicolumn{3}{c}{} & \multicolumn{3}{c}{$n=30$} & \multicolumn{3}{c}{$n=40$} & \multicolumn{3}{c}{$n=50$}\\
 \cmidrule(l){4-6} \cmidrule(l){7-9} \cmidrule(l){10-12}
 & & & q=2 & q=6 & q=10 & q=2 & q=6 & q=10 & q=2 & q=6 & q=10\\
 \midrule 
 \midrule\multirow{6}{*}{\textbf{AR}} & \multirow{3}{*}{\textbf{$\alpha=0.05$}} & \textbf{Ex. 1} & 0.70 & 0.12 & 0.00 & 0.79 & 0.24 & 0.05 & 0.84 & 0.32 & 0.09 \\ 
   &  & \textbf{Ex. 2} & 0.95 & 0.53 & 0.09 & 0.98 & 0.83 & 0.54 & 1.00 & 0.96 & 0.88 \\ 
    &  & \textbf{Ex. 3} & 0.80 & 0.25 & 0.03 & 0.90 & 0.46 & 0.19 & 0.96 & 0.68 & 0.44 \\ 
   \cdashline{2-12} & \multirow{3}{*}{\textbf{$\alpha=0.1$}} & \textbf{Ex. 1} & 0.82 & 0.30 & 0.09 & 0.88 & 0.45 & 0.22 & 0.91 & 0.48 & 0.30 \\ 
   &  & \textbf{Ex. 2} & 0.98 & 0.73 & 0.49 & 0.99 & 0.91 & 0.83 & 1.00 & 0.98 & 0.95 \\ 
   &  & \textbf{Ex. 3} & 0.89 & 0.48 & 0.24 & 0.96 & 0.67 & 0.48 & 0.98 & 0.82 & 0.67 \\ 
  \midrule\multirow{6}{*}{\textbf{MA}} & \multirow{3}{*}{\textbf{$\alpha=0.05$}} & \textbf{Ex. 1} & 0.69 & 0.20 & 0.04 & 0.80 & 0.34 & 0.11 & 0.85 & 0.40 & 0.17 \\ 
   &  & \textbf{Ex. 2} & 0.94 & 0.63 & 0.37 & 0.98 & 0.86 & 0.71 & 1.00 & 0.97 & 0.93 \\ 
   &  & \textbf{Ex. 3} & 0.80 & 0.36 & 0.14 & 0.91 & 0.57 & 0.35 & 0.97 & 0.74 & 0.56 \\ 
   \cdashline{2-12} & \multirow{3}{*}{\textbf{$\alpha=0.1$}} & \textbf{Ex. 1} & 0.79 & 0.35 & 0.16 & 0.87 & 0.52 & 0.28 & 0.91 & 0.53 & 0.36 \\ 
   &  & \textbf{Ex. 2} & 0.97 & 0.77 & 0.60 & 0.99 & 0.92 & 0.86 & 1.00 & 0.98 & 0.96 \\ 
   &  & \textbf{Ex. 3} & 0.88 & 0.54 & 0.33 & 0.95 & 0.69 & 0.55 & 0.98 & 0.84 & 0.72 \\ 
   \bottomrule 
\end{tabular}
}
\end{table}

\begin{table}[ht]
\centering
\caption{Change in Mean, $\hat{\mu}=\tilde{F}_n$, a=0.5, $\Delta=0.5$} 
\label{table2DimBootstrapWChEst_a0.5.txt_Delta0.5}
{\footnotesize
\begin{tabular}{lllrrr|rrr|rrr}
  \toprule 
 \multicolumn{3}{c}{} & \multicolumn{3}{c}{$n=30$} & \multicolumn{3}{c}{$n=40$} & \multicolumn{3}{c}{$n=50$}\\
 \cmidrule(l){4-6} \cmidrule(l){7-9} \cmidrule(l){10-12}
 & & & q=2 & q=6 & q=10 & q=2 & q=6 & q=10 & q=2 & q=6 & q=10\\
 \midrule 
 \midrule\multirow{6}{*}{\textbf{AR}} & \multirow{3}{*}{\textbf{$\alpha=0.05$}} & \textbf{Ex. 1} & 0.79 & 0.46 & 0.38 & 0.84 & 0.57 & 0.47 & 0.87 & 0.56 & 0.46 \\ 
   &  & \textbf{Ex. 2} & 0.97 & 0.82 & 0.74 & 0.98 & 0.93 & 0.89 & 1.00 & 0.99 & 0.97 \\ 
    &  & \textbf{Ex. 3} & 0.86 & 0.58 & 0.48 & 0.93 & 0.70 & 0.61 & 0.98 & 0.83 & 0.74 \\ 
   \cdashline{2-12} & \multirow{3}{*}{\textbf{$\alpha=0.1$}} & \textbf{Ex. 1} & 0.86 & 0.63 & 0.53 & 0.91 & 0.70 & 0.61 & 0.93 & 0.70 & 0.62 \\ 
   &  & \textbf{Ex. 2} & 0.99 & 0.91 & 0.86 & 1.00 & 0.96 & 0.94 & 1.00 & 0.99 & 0.99 \\ 
   &  & \textbf{Ex. 3} & 0.92 & 0.71 & 0.64 & 0.96 & 0.81 & 0.74 & 0.98 & 0.89 & 0.83 \\ 
  \midrule\multirow{6}{*}{\textbf{MA}} & \multirow{3}{*}{\textbf{$\alpha=0.05$}} & \textbf{Ex. 1} & 0.77 & 0.47 & 0.36 & 0.84 & 0.55 & 0.45 & 0.87 & 0.54 & 0.45 \\ 
   &  & \textbf{Ex. 2} & 0.97 & 0.82 & 0.74 & 0.98 & 0.93 & 0.89 & 1.00 & 0.98 & 0.97 \\ 
   &  & \textbf{Ex. 3} & 0.84 & 0.58 & 0.47 & 0.93 & 0.69 & 0.61 & 0.97 & 0.81 & 0.72 \\ 
   \cdashline{2-12} & \multirow{3}{*}{\textbf{$\alpha=0.1$}} & \textbf{Ex. 1} & 0.85 & 0.61 & 0.51 & 0.90 & 0.68 & 0.60 & 0.94 & 0.68 & 0.59 \\ 
   &  & \textbf{Ex. 2} & 0.98 & 0.90 & 0.85 & 1.00 & 0.96 & 0.94 & 1.00 & 0.99 & 0.99 \\ 
   &  & \textbf{Ex. 3} & 0.91 & 0.71 & 0.61 & 0.96 & 0.79 & 0.73 & 0.98 & 0.88 & 0.83 \\ 
   \bottomrule 
\end{tabular}
}
\end{table}

\begin{table}[ht]
\centering
\caption{Change in Mean, $\hat{\mu}=F_n$, a=0.2, $\Delta=1$} 
\label{table2DimBootstrap_a0.2.txt_Delta1}
{\footnotesize
\begin{tabular}{lllrrr|rrr|rrr}
  \toprule 
 \multicolumn{3}{c}{} & \multicolumn{3}{c}{$n=30$} & \multicolumn{3}{c}{$n=40$} & \multicolumn{3}{c}{$n=50$}\\
 \cmidrule(l){4-6} \cmidrule(l){7-9} \cmidrule(l){10-12}
 & & & q=2 & q=6 & q=10 & q=2 & q=6 & q=10 & q=2 & q=6 & q=10\\
 \midrule 
 \midrule\multirow{6}{*}{\textbf{AR}} & \multirow{3}{*}{\textbf{$\alpha=0.05$}} & \textbf{Ex. 1} & 0.90 & 0.04 & 0.00 & 1.00 & 0.32 & 0.01 & 1.00 & 0.81 & 0.02 \\ 
   &  & \textbf{Ex. 2} & 1.00 & 1.00 & 0.99 & 1.00 & 1.00 & 1.00 & 1.00 & 1.00 & 1.00 \\ 
    &  & \textbf{Ex. 3} & 1.00 & 1.00 & 0.78 & 1.00 & 1.00 & 1.00 & 1.00 & 1.00 & 1.00 \\ 
   \cdashline{2-12} & \multirow{3}{*}{\textbf{$\alpha=0.1$}} & \textbf{Ex. 1} & 0.97 & 0.39 & 0.03 & 1.00 & 0.87 & 0.23 & 1.00 & 1.00 & 0.54 \\ 
   &  & \textbf{Ex. 2} & 1.00 & 1.00 & 1.00 & 1.00 & 1.00 & 1.00 & 1.00 & 1.00 & 1.00 \\ 
   &  & \textbf{Ex. 3} & 1.00 & 1.00 & 1.00 & 1.00 & 1.00 & 1.00 & 1.00 & 1.00 & 1.00 \\ 
  \midrule\multirow{6}{*}{\textbf{MA}} & \multirow{3}{*}{\textbf{$\alpha=0.05$}} & \textbf{Ex. 1} & 0.92 & 0.20 & 0.01 & 1.00 & 0.80 & 0.10 & 1.00 & 1.00 & 0.39 \\ 
   &  & \textbf{Ex. 2} & 1.00 & 1.00 & 1.00 & 1.00 & 1.00 & 1.00 & 1.00 & 1.00 & 1.00 \\ 
   &  & \textbf{Ex. 3} & 1.00 & 1.00 & 0.99 & 1.00 & 1.00 & 1.00 & 1.00 & 1.00 & 1.00 \\ 
   \cdashline{2-12} & \multirow{3}{*}{\textbf{$\alpha=0.1$}} & \textbf{Ex. 1} & 0.97 & 0.58 & 0.13 & 1.00 & 0.98 & 0.59 & 1.00 & 1.00 & 0.97 \\ 
   &  & \textbf{Ex. 2} & 1.00 & 1.00 & 1.00 & 1.00 & 1.00 & 1.00 & 1.00 & 1.00 & 1.00 \\ 
   &  & \textbf{Ex. 3} & 1.00 & 1.00 & 1.00 & 1.00 & 1.00 & 1.00 & 1.00 & 1.00 & 1.00 \\ 
   \bottomrule 
\end{tabular}
}
\end{table}

\begin{table}[ht]
\centering
\caption{Change in Mean, $\hat{\mu}=\tilde{F}_n$, a=0.2, $\Delta=1$} 
\label{table2DimBootstrapWChEst_a0.2.txt_Delta1}
{\footnotesize
\begin{tabular}{lllrrr|rrr|rrr}
  \toprule 
 \multicolumn{3}{c}{} & \multicolumn{3}{c}{$n=30$} & \multicolumn{3}{c}{$n=40$} & \multicolumn{3}{c}{$n=50$}\\
 \cmidrule(l){4-6} \cmidrule(l){7-9} \cmidrule(l){10-12}
 & & & q=2 & q=6 & q=10 & q=2 & q=6 & q=10 & q=2 & q=6 & q=10\\
 \midrule 
 \midrule\multirow{6}{*}{\textbf{AR}} & \multirow{3}{*}{\textbf{$\alpha=0.05$}} & \textbf{Ex. 1} & 0.97 & 0.87 & 0.72 & 1.00 & 0.99 & 0.95 & 1.00 & 1.00 & 1.00 \\ 
   &  & \textbf{Ex. 2} & 1.00 & 1.00 & 1.00 & 1.00 & 1.00 & 1.00 & 1.00 & 1.00 & 1.00 \\ 
    &  & \textbf{Ex. 3} & 1.00 & 1.00 & 1.00 & 1.00 & 1.00 & 1.00 & 1.00 & 1.00 & 1.00 \\ 
   \cdashline{2-12} & \multirow{3}{*}{\textbf{$\alpha=0.1$}} & \textbf{Ex. 1} & 0.99 & 0.95 & 0.91 & 1.00 & 1.00 & 1.00 & 1.00 & 1.00 & 1.00 \\ 
   &  & \textbf{Ex. 2} & 1.00 & 1.00 & 1.00 & 1.00 & 1.00 & 1.00 & 1.00 & 1.00 & 1.00 \\ 
   &  & \textbf{Ex. 3} & 1.00 & 1.00 & 1.00 & 1.00 & 1.00 & 1.00 & 1.00 & 1.00 & 1.00 \\ 
  \midrule\multirow{6}{*}{\textbf{MA}} & \multirow{3}{*}{\textbf{$\alpha=0.05$}} & \textbf{Ex. 1} & 0.97 & 0.87 & 0.73 & 1.00 & 1.00 & 0.97 & 1.00 & 1.00 & 1.00 \\ 
   &  & \textbf{Ex. 2} & 1.00 & 1.00 & 1.00 & 1.00 & 1.00 & 1.00 & 1.00 & 1.00 & 1.00 \\ 
   &  & \textbf{Ex. 3} & 1.00 & 1.00 & 1.00 & 1.00 & 1.00 & 1.00 & 1.00 & 1.00 & 1.00 \\ 
   \cdashline{2-12} & \multirow{3}{*}{\textbf{$\alpha=0.1$}} & \textbf{Ex. 1} & 0.99 & 0.94 & 0.90 & 1.00 & 1.00 & 1.00 & 1.00 & 1.00 & 1.00 \\ 
   &  & \textbf{Ex. 2} & 1.00 & 1.00 & 1.00 & 1.00 & 1.00 & 1.00 & 1.00 & 1.00 & 1.00 \\ 
   &  & \textbf{Ex. 3} & 1.00 & 1.00 & 1.00 & 1.00 & 1.00 & 1.00 & 1.00 & 1.00 & 1.00 \\ 
   \bottomrule 
\end{tabular}
}
\end{table}

\begin{table}[ht]
\centering
\caption{Change in Mean, $\hat{\mu}=F_n$, a=0.5, $\Delta=1$} 
\label{table2DimBootstrap_a0.5.txt_Delta1}
{\footnotesize
\begin{tabular}{lllrrr|rrr|rrr}
  \toprule 
 \multicolumn{3}{c}{} & \multicolumn{3}{c}{$n=30$} & \multicolumn{3}{c}{$n=40$} & \multicolumn{3}{c}{$n=50$}\\
 \cmidrule(l){4-6} \cmidrule(l){7-9} \cmidrule(l){10-12}
 & & & q=2 & q=6 & q=10 & q=2 & q=6 & q=10 & q=2 & q=6 & q=10\\
 \midrule 
 \midrule\multirow{6}{*}{\textbf{AR}} & \multirow{3}{*}{\textbf{$\alpha=0.05$}} & \textbf{Ex. 1} & 0.86 & 0.15 & 0.00 & 0.97 & 0.39 & 0.04 & 0.99 & 0.59 & 0.10 \\ 
   &  & \textbf{Ex. 2} & 1.00 & 1.00 & 0.75 & 1.00 & 1.00 & 1.00 & 1.00 & 1.00 & 1.00 \\ 
    &  & \textbf{Ex. 3} & 1.00 & 0.76 & 0.26 & 1.00 & 0.98 & 0.86 & 1.00 & 1.00 & 1.00 \\ 
   \cdashline{2-12} & \multirow{3}{*}{\textbf{$\alpha=0.1$}} & \textbf{Ex. 1} & 0.93 & 0.43 & 0.11 & 0.99 & 0.66 & 0.30 & 1.00 & 0.83 & 0.46 \\ 
   &  & \textbf{Ex. 2} & 1.00 & 1.00 & 1.00 & 1.00 & 1.00 & 1.00 & 1.00 & 1.00 & 1.00 \\ 
   &  & \textbf{Ex. 3} & 1.00 & 0.92 & 0.75 & 1.00 & 1.00 & 0.98 & 1.00 & 1.00 & 1.00 \\ 
  \midrule\multirow{6}{*}{\textbf{MA}} & \multirow{3}{*}{\textbf{$\alpha=0.05$}} & \textbf{Ex. 1} & 0.86 & 0.30 & 0.04 & 0.97 & 0.59 & 0.16 & 0.99 & 0.77 & 0.33 \\ 
   &  & \textbf{Ex. 2} & 1.00 & 1.00 & 0.98 & 1.00 & 1.00 & 1.00 & 1.00 & 1.00 & 1.00 \\ 
   &  & \textbf{Ex. 3} & 1.00 & 0.88 & 0.59 & 1.00 & 0.99 & 0.95 & 1.00 & 1.00 & 1.00 \\ 
   \cdashline{2-12} & \multirow{3}{*}{\textbf{$\alpha=0.1$}} & \textbf{Ex. 1} & 0.93 & 0.53 & 0.23 & 0.99 & 0.75 & 0.48 & 0.99 & 0.90 & 0.65 \\ 
   &  & \textbf{Ex. 2} & 1.00 & 1.00 & 1.00 & 1.00 & 1.00 & 1.00 & 1.00 & 1.00 & 1.00 \\ 
   &  & \textbf{Ex. 3} & 1.00 & 0.95 & 0.85 & 1.00 & 1.00 & 0.99 & 1.00 & 1.00 & 1.00 \\ 
   \bottomrule 
\end{tabular}
}
\end{table}

\begin{table}[ht]
\centering
\caption{Change in Mean, $\hat{\mu}=\tilde{F}_n$, a=0.5, $\Delta=1$} 
\label{table2DimBootstrapWChEst_a0.5.txt_Delta1}
{\footnotesize
\begin{tabular}{lllrrr|rrr|rrr}
  \toprule 
 \multicolumn{3}{c}{} & \multicolumn{3}{c}{$n=30$} & \multicolumn{3}{c}{$n=40$} & \multicolumn{3}{c}{$n=50$}\\
 \cmidrule(l){4-6} \cmidrule(l){7-9} \cmidrule(l){10-12}
 & & & q=2 & q=6 & q=10 & q=2 & q=6 & q=10 & q=2 & q=6 & q=10\\
 \midrule 
 \midrule\multirow{6}{*}{\textbf{AR}} & \multirow{3}{*}{\textbf{$\alpha=0.05$}} & \textbf{Ex. 1} & 0.93 & 0.67 & 0.56 & 0.98 & 0.83 & 0.73 & 0.99 & 0.92 & 0.83 \\ 
   &  & \textbf{Ex. 2} & 1.00 & 1.00 & 1.00 & 1.00 & 1.00 & 1.00 & 1.00 & 1.00 & 1.00 \\ 
    &  & \textbf{Ex. 3} & 1.00 & 0.95 & 0.88 & 1.00 & 1.00 & 0.98 & 1.00 & 1.00 & 1.00 \\ 
   \cdashline{2-12} & \multirow{3}{*}{\textbf{$\alpha=0.1$}} & \textbf{Ex. 1} & 0.96 & 0.80 & 0.73 & 1.00 & 0.92 & 0.87 & 1.00 & 0.98 & 0.91 \\ 
   &  & \textbf{Ex. 2} & 1.00 & 1.00 & 1.00 & 1.00 & 1.00 & 1.00 & 1.00 & 1.00 & 1.00 \\ 
   &  & \textbf{Ex. 3} & 1.00 & 0.99 & 0.96 & 1.00 & 1.00 & 1.00 & 1.00 & 1.00 & 1.00 \\ 
  \midrule\multirow{6}{*}{\textbf{MA}} & \multirow{3}{*}{\textbf{$\alpha=0.05$}} & \textbf{Ex. 1} & 0.92 & 0.66 & 0.53 & 0.98 & 0.83 & 0.73 & 0.99 & 0.93 & 0.82 \\ 
   &  & \textbf{Ex. 2} & 1.00 & 1.00 & 1.00 & 1.00 & 1.00 & 1.00 & 1.00 & 1.00 & 1.00 \\ 
   &  & \textbf{Ex. 3} & 1.00 & 0.95 & 0.89 & 1.00 & 1.00 & 0.99 & 1.00 & 1.00 & 1.00 \\ 
   \cdashline{2-12} & \multirow{3}{*}{\textbf{$\alpha=0.1$}} & \textbf{Ex. 1} & 0.96 & 0.80 & 0.70 & 0.99 & 0.92 & 0.84 & 1.00 & 0.98 & 0.93 \\ 
   &  & \textbf{Ex. 2} & 1.00 & 1.00 & 1.00 & 1.00 & 1.00 & 1.00 & 1.00 & 1.00 & 1.00 \\ 
   &  & \textbf{Ex. 3} & 1.00 & 0.99 & 0.96 & 1.00 & 1.00 & 1.00 & 1.00 & 1.00 & 1.00 \\ 
   \bottomrule 
\end{tabular}
}
\end{table}

\begin{table}[ht]
\centering
\caption{Change in Skewness, $\hat{\mu}=F_n$, a=0.2} 
\label{table2DimBootstrapSkewness_a0.2.txt}
{\footnotesize
\begin{tabular}{lllrrr|rrr|rrr}
  \toprule 
 \multicolumn{3}{c}{} & \multicolumn{3}{c}{$n=30$} & \multicolumn{3}{c}{$n=40$} & \multicolumn{3}{c}{$n=50$}\\
 \cmidrule(l){4-6} \cmidrule(l){7-9} \cmidrule(l){10-12}
 & & & q=2 & q=6 & q=10 & q=2 & q=6 & q=10 & q=2 & q=6 & q=10\\
 \midrule 
 \midrule\multirow{6}{*}{\textbf{AR}} & \multirow{3}{*}{\textbf{$\alpha=0.05$}} & \textbf{Ex. 1} & 0.28 & 0.01 & 0.00 & 0.80 & 0.15 & 0.00 & 0.99 & 0.56 & 0.02 \\ 
   &  & \textbf{Ex. 2} & 1.00 & 1.00 & 0.99 & 1.00 & 1.00 & 1.00 & 1.00 & 1.00 & 1.00 \\ 
    &  & \textbf{Ex. 3} & 1.00 & 0.91 & 0.37 & 1.00 & 1.00 & 1.00 & 1.00 & 1.00 & 1.00 \\ 
   \cdashline{2-12} & \multirow{3}{*}{\textbf{$\alpha=0.1$}} & \textbf{Ex. 1} & 0.49 & 0.15 & 0.02 & 0.93 & 0.55 & 0.11 & 1.00 & 0.94 & 0.36 \\ 
   &  & \textbf{Ex. 2} & 1.00 & 1.00 & 1.00 & 1.00 & 1.00 & 1.00 & 1.00 & 1.00 & 1.00 \\ 
   &  & \textbf{Ex. 3} & 1.00 & 0.99 & 0.92 & 1.00 & 1.00 & 1.00 & 1.00 & 1.00 & 1.00 \\ 
  \midrule\multirow{6}{*}{\textbf{MA}} & \multirow{3}{*}{\textbf{$\alpha=0.05$}} & \textbf{Ex. 1} & 0.31 & 0.07 & 0.00 & 0.82 & 0.41 & 0.04 & 0.99 & 0.89 & 0.23 \\ 
   &  & \textbf{Ex. 2} & 1.00 & 1.00 & 1.00 & 1.00 & 1.00 & 1.00 & 1.00 & 1.00 & 1.00 \\ 
   &  & \textbf{Ex. 3} & 1.00 & 0.98 & 0.83 & 1.00 & 1.00 & 1.00 & 1.00 & 1.00 & 1.00 \\ 
   \cdashline{2-12} & \multirow{3}{*}{\textbf{$\alpha=0.1$}} & \textbf{Ex. 1} & 0.51 & 0.23 & 0.07 & 0.93 & 0.73 & 0.31 & 1.00 & 0.98 & 0.80 \\ 
   &  & \textbf{Ex. 2} & 1.00 & 1.00 & 1.00 & 1.00 & 1.00 & 1.00 & 1.00 & 1.00 & 1.00 \\ 
   &  & \textbf{Ex. 3} & 1.00 & 1.00 & 0.98 & 1.00 & 1.00 & 1.00 & 1.00 & 1.00 & 1.00 \\ 
   \bottomrule 
\end{tabular}
}
\end{table}

\begin{table}[ht]
\centering
\caption{Change in Skewness, $\hat{\mu}=\tilde{F}_n$, a=0.2} 
\label{table2DimBootstrapSkewnessChEst_a0.2.txt}
{\footnotesize
\begin{tabular}{lllrrr|rrr|rrr}
  \toprule 
 \multicolumn{3}{c}{} & \multicolumn{3}{c}{$n=30$} & \multicolumn{3}{c}{$n=40$} & \multicolumn{3}{c}{$n=50$}\\
 \cmidrule(l){4-6} \cmidrule(l){7-9} \cmidrule(l){10-12}
 & & & q=2 & q=6 & q=10 & q=2 & q=6 & q=10 & q=2 & q=6 & q=10\\
 \midrule 
 \midrule\multirow{6}{*}{\textbf{AR}} & \multirow{3}{*}{\textbf{$\alpha=0.05$}} & \textbf{Ex. 1} & 0.45 & 0.35 & 0.28 & 0.88 & 0.80 & 0.69 & 0.99 & 0.98 & 0.93 \\ 
   &  & \textbf{Ex. 2} & 1.00 & 1.00 & 1.00 & 1.00 & 1.00 & 1.00 & 1.00 & 1.00 & 1.00 \\ 
    &  & \textbf{Ex. 3} & 1.00 & 0.99 & 0.97 & 1.00 & 1.00 & 1.00 & 1.00 & 1.00 & 1.00 \\ 
   \cdashline{2-12} & \multirow{3}{*}{\textbf{$\alpha=0.1$}} & \textbf{Ex. 1} & 0.62 & 0.54 & 0.50 & 0.95 & 0.91 & 0.87 & 1.00 & 0.99 & 0.99 \\ 
   &  & \textbf{Ex. 2} & 1.00 & 1.00 & 1.00 & 1.00 & 1.00 & 1.00 & 1.00 & 1.00 & 1.00 \\ 
   &  & \textbf{Ex. 3} & 1.00 & 1.00 & 1.00 & 1.00 & 1.00 & 1.00 & 1.00 & 1.00 & 1.00 \\ 
  \midrule\multirow{6}{*}{\textbf{MA}} & \multirow{3}{*}{\textbf{$\alpha=0.05$}} & \textbf{Ex. 1} & 0.43 & 0.36 & 0.28 & 0.88 & 0.82 & 0.74 & 0.99 & 0.98 & 0.96 \\ 
   &  & \textbf{Ex. 2} & 1.00 & 1.00 & 1.00 & 1.00 & 1.00 & 1.00 & 1.00 & 1.00 & 1.00 \\ 
   &  & \textbf{Ex. 3} & 1.00 & 1.00 & 0.98 & 1.00 & 1.00 & 1.00 & 1.00 & 1.00 & 1.00 \\ 
   \cdashline{2-12} & \multirow{3}{*}{\textbf{$\alpha=0.1$}} & \textbf{Ex. 1} & 0.60 & 0.54 & 0.47 & 0.95 & 0.92 & 0.88 & 1.00 & 0.99 & 0.99 \\ 
   &  & \textbf{Ex. 2} & 1.00 & 1.00 & 1.00 & 1.00 & 1.00 & 1.00 & 1.00 & 1.00 & 1.00 \\ 
   &  & \textbf{Ex. 3} & 1.00 & 1.00 & 1.00 & 1.00 & 1.00 & 1.00 & 1.00 & 1.00 & 1.00 \\ 
   \bottomrule 
\end{tabular}
}
\end{table}

\begin{table}[ht]
\centering
\caption{Change in Skewness, $\hat{\mu}=F_n$, a=0.5} 
\label{table2DimBootstrapSkewness_a0.5.txt}
{\footnotesize
\begin{tabular}{lllrrr|rrr|rrr}
  \toprule 
 \multicolumn{3}{c}{} & \multicolumn{3}{c}{$n=30$} & \multicolumn{3}{c}{$n=40$} & \multicolumn{3}{c}{$n=50$}\\
 \cmidrule(l){4-6} \cmidrule(l){7-9} \cmidrule(l){10-12}
 & & & q=2 & q=6 & q=10 & q=2 & q=6 & q=10 & q=2 & q=6 & q=10\\
 \midrule 
 \midrule\multirow{6}{*}{\textbf{AR}} & \multirow{3}{*}{\textbf{$\alpha=0.05$}} & \textbf{Ex. 1} & 0.38 & 0.03 & 0.00 & 0.68 & 0.15 & 0.00 & 0.85 & 0.32 & 0.05 \\ 
   &  & \textbf{Ex. 2} & 1.00 & 1.00 & 0.82 & 1.00 & 1.00 & 1.00 & 1.00 & 1.00 & 1.00 \\ 
    &  & \textbf{Ex. 3} & 0.89 & 0.50 & 0.11 & 1.00 & 0.96 & 0.82 & 1.00 & 1.00 & 1.00 \\ 
   \cdashline{2-12} & \multirow{3}{*}{\textbf{$\alpha=0.1$}} & \textbf{Ex. 1} & 0.53 & 0.16 & 0.04 & 0.82 & 0.38 & 0.13 & 0.95 & 0.64 & 0.27 \\ 
   &  & \textbf{Ex. 2} & 1.00 & 1.00 & 0.99 & 1.00 & 1.00 & 1.00 & 1.00 & 1.00 & 1.00 \\ 
   &  & \textbf{Ex. 3} & 0.96 & 0.78 & 0.55 & 1.00 & 1.00 & 0.97 & 1.00 & 1.00 & 1.00 \\ 
  \midrule\multirow{6}{*}{\textbf{MA}} & \multirow{3}{*}{\textbf{$\alpha=0.05$}} & \textbf{Ex. 1} & 0.36 & 0.09 & 0.01 & 0.67 & 0.26 & 0.06 & 0.83 & 0.52 & 0.18 \\ 
   &  & \textbf{Ex. 2} & 1.00 & 1.00 & 0.98 & 1.00 & 1.00 & 1.00 & 1.00 & 1.00 & 1.00 \\ 
   &  & \textbf{Ex. 3} & 0.88 & 0.63 & 0.37 & 1.00 & 0.98 & 0.93 & 1.00 & 1.00 & 1.00 \\ 
   \cdashline{2-12} & \multirow{3}{*}{\textbf{$\alpha=0.1$}} & \textbf{Ex. 1} & 0.51 & 0.22 & 0.07 & 0.81 & 0.49 & 0.22 & 0.93 & 0.72 & 0.44 \\ 
   &  & \textbf{Ex. 2} & 1.00 & 1.00 & 1.00 & 1.00 & 1.00 & 1.00 & 1.00 & 1.00 & 1.00 \\ 
   &  & \textbf{Ex. 3} & 0.96 & 0.83 & 0.67 & 1.00 & 1.00 & 0.99 & 1.00 & 1.00 & 1.00 \\ 
   \bottomrule 
\end{tabular}
}
\end{table}

\begin{table}[ht]
\centering
\caption{Change in Skewness, $\hat{\mu}=\tilde{F}_n$, a=0.5} 
\label{table2DimBootstrapSkewnessChEst_a0.5.txt}
{\footnotesize
\begin{tabular}{lllrrr|rrr|rrr}
  \toprule 
 \multicolumn{3}{c}{} & \multicolumn{3}{c}{$n=30$} & \multicolumn{3}{c}{$n=40$} & \multicolumn{3}{c}{$n=50$}\\
 \cmidrule(l){4-6} \cmidrule(l){7-9} \cmidrule(l){10-12}
 & & & q=2 & q=6 & q=10 & q=2 & q=6 & q=10 & q=2 & q=6 & q=10\\
 \midrule 
 \midrule\multirow{6}{*}{\textbf{AR}} & \multirow{3}{*}{\textbf{$\alpha=0.05$}} & \textbf{Ex. 1} & 0.50 & 0.30 & 0.23 & 0.76 & 0.54 & 0.41 & 0.89 & 0.72 & 0.64 \\ 
   &  & \textbf{Ex. 2} & 1.00 & 1.00 & 0.99 & 1.00 & 1.00 & 1.00 & 1.00 & 1.00 & 1.00 \\ 
    &  & \textbf{Ex. 3} & 0.92 & 0.78 & 0.67 & 1.00 & 0.99 & 0.96 & 1.00 & 1.00 & 1.00 \\ 
   \cdashline{2-12} & \multirow{3}{*}{\textbf{$\alpha=0.1$}} & \textbf{Ex. 1} & 0.65 & 0.46 & 0.39 & 0.87 & 0.73 & 0.65 & 0.96 & 0.84 & 0.79 \\ 
   &  & \textbf{Ex. 2} & 1.00 & 1.00 & 1.00 & 1.00 & 1.00 & 1.00 & 1.00 & 1.00 & 1.00 \\ 
   &  & \textbf{Ex. 3} & 0.97 & 0.89 & 0.83 & 1.00 & 1.00 & 1.00 & 1.00 & 1.00 & 1.00 \\ 
  \midrule\multirow{6}{*}{\textbf{MA}} & \multirow{3}{*}{\textbf{$\alpha=0.05$}} & \textbf{Ex. 1} & 0.45 & 0.30 & 0.24 & 0.73 & 0.54 & 0.44 & 0.87 & 0.72 & 0.64 \\ 
   &  & \textbf{Ex. 2} & 1.00 & 1.00 & 1.00 & 1.00 & 1.00 & 1.00 & 1.00 & 1.00 & 1.00 \\ 
   &  & \textbf{Ex. 3} & 0.92 & 0.79 & 0.70 & 1.00 & 1.00 & 0.97 & 1.00 & 1.00 & 1.00 \\ 
   \cdashline{2-12} & \multirow{3}{*}{\textbf{$\alpha=0.1$}} & \textbf{Ex. 1} & 0.61 & 0.45 & 0.36 & 0.85 & 0.70 & 0.65 & 0.95 & 0.84 & 0.77 \\ 
   &  & \textbf{Ex. 2} & 1.00 & 1.00 & 1.00 & 1.00 & 1.00 & 1.00 & 1.00 & 1.00 & 1.00 \\ 
   &  & \textbf{Ex. 3} & 0.96 & 0.88 & 0.84 & 1.00 & 1.00 & 1.00 & 1.00 & 1.00 & 1.00 \\ 
   \bottomrule 
\end{tabular}
}
\end{table}
\begin{subsection}{Conclusion}
In conclusion, the simulations show that the proposed tests display the typical over-rejection property of bootstrap tests but have good empirical power against changes in the distribution. The latter is strongly influenced by the size of the set on which there is a change. While the two considered bootstrap procedures (\textbf{MA} and \textbf{AR}) show comparable results, the choice of the bandwidth has a significant effect, with smaller bandwidths leading to higher rejection rates. 
In comparison to $\hat{\mu}=F_n$, the estimator $\hat{\mu}=\tilde{F}_n$ has worse adherence to the nominal level under the null hypothesis but also better power against changes in mean or in the skewness. This might be due to the fact that $\tilde{F}_n$ is a more accurate estimator for the mean under the alternative but performs slightly worse under the null hypothesis.
\end{subsection}
\end{section}

\FloatBarrier

\begin{section}{Proofs}

\begin{subsection}{Preliminary results}
\begin{lemma}\label{le:RosZhang}
Let $\{X_\mbf{k}\}_{\mbf{k}\in\N^d}$ be an $H$-valued centered random field with $\lim\limits_{\tau\to\infty}\rho_{\R}(\tau)<1$. Then for any $r\geq 2$, there exists a positive constant $B_{d,r}$ depending only on $r$, $d$ and $\rho_\R(\cdot)$ such that for any finite set $S\subset \N^d$,
\begin{equation}
\mathrm{E}\left\|\sum_{\mbf{k}\in S} X_\mbf{k}\right\|^r \leq B_{d,r} \left\{ \sum_{\mbf{k}\in S}\mathrm{E}\|X_\mbf{k}\|^r + \left( \sum_{\mbf{k}\in S}\mathrm{E}\|X_\mbf{k}\|^2\right)^{r/2}\right\}.
\label{eq:rosenthalIneq}
\end{equation}
If $\sup\limits_{\mbf{k}\in\N^d}\mathrm{E}\|X_{\mbf{k}}\|^{r}<\infty$, this implies
\begin{equation}
\mathrm{E}\left\|\sum_{\mbf{k}\in S} X_\mbf{k}\right\|^r \leq B_{d,r} \left\{\sup\limits_{\mbf{k}\in\N^d}\mathrm{E}\|X_{\mbf{k}}\|^{r} + (\sup\limits_{\mbf{k}\in\N^d}\mathrm{E}\|X_{\mbf{k}}\|^2)^{r/2}\right\} (\#S)^{r/2} =: C(d,r,X) (\#S)^{r/2}.
\label{eq:momIneq}
\end{equation}
If \eqref{eq:momIneq} holds for $r>2$ and every block $S$ in $\Z^d$, $U$ is any block in $\Z^d$, and
\begin{equation}
M(U)=\max\limits_{W\triangleleft U}\left\|\sum\limits_{\mbf{j}\in W}X_\mbf{j}\right\|,
\label{eq:defMax}
\end{equation}
then 
\begin{equation}
\mathrm{E}\left[M(U)^r\right] \leq \tilde{C} C(d,r,X) (\#U)^{r/2}
\label{eq:maxIneq}
\end{equation}
with $\tilde{C}=\left(\frac{5}{2}\right)^{d} (1-2^{(1-\frac{r}{2})/r})^{-dr}$.
\end{lemma}
\begin{remark}
For $H$-valued processes, an alternative definition of $\rho$-mixing is given by the coefficients
\begin{equation*}
\resizebox{\hsize}{!}{$\begin{aligned}\rho_{H}(\mathcal{A},\mathcal{B})&=\sup\left\{\frac{|\mathrm{E}[\sprod{X,Y}] - \sprod{\mathrm{E}X,\mathrm{E}Y}|}{\|X\|_2 \|Y\|_2}:\; X\in L^2(\mathcal{A},H), Y \in L^2(\mathcal{B},H),\|X\|_2,\|Y\|_2>0 \right\}\\
&= \sup\left\{\frac{|\mathrm{E}[\sprod{X,Y}]|}{\|X\|_2 \|Y\|_2}:\; X\in L^2(\mathcal{A},H), Y \in L^2(\mathcal{B},H),\|X\|_2,\|Y\|_2>0, \mathrm{E}X=\mathrm{E}Y=0 \right\}.\end{aligned}$}
\end{equation*}
(Note that the above equality is a consequence of the well known inequality $\|X\|_2\geq \|X-\mathrm{E}X\|_2$.) Analogously to the real-valued case, one can then define $\rho_{H}(r)$ and $\rho_{H}^*(r)$ for random fields. As shown in \cite{bradley1985}, Theorem 4.2, the coefficients $\rho_{H}$ and $\rho_\R$ coincide and therefore $\rho_{H}(\cdot)=\rho_\R(\cdot)$ and $\rho_{H}^*(\cdot)=\rho_{\R}^*(\cdot)$. 
\end{remark}
\begin{proof}
We prove \eqref{eq:rosenthalIneq} by induction over $d$. The induction start follows from Theorem 2 in \cite{zhang1998}. For the application of the theorem, note that for $d=1$, the definitions of $\rho_{\R}$- and $\rho_{\R}^*$-mixing coincide, and that for $\rho_{\R}^*=\rho_{H}^*$ instead of $\rho_{\R}=\rho_H$, \eqref{eq:rosenthalIneq} is Theorem 2 in \cite{zhang1998}. Let $d\geq 2$ and assume that \eqref{eq:rosenthalIneq} holds for any dimension smaller than $d$. We use an analogous argumentation to \cite{bradley2007}, Volume III, p.234, to prove that \eqref{eq:rosenthalIneq} holds for $d$ itself.
For any nonempty finite set $S\subseteq\N^d$ and $j\in\N$, define sets $S(j)=\{\mbf{k}\in S:\, k_1=j\}$, $T(j)=\{\mbf{k}\in\N^d:\, k_1=j\}$ and $N(S)=\{j\in\N:\, S(j)\neq \emptyset\}$. Set $Y_j=\sum\limits_{\mbf{k}\in S(j)}X_\mbf{k}$ if $j\in N(S)$ and $Y_j=0$ otherwise. Then $Y_j$ is a measurable transform of $\{X_\mbf{k}\}_{\mbf{k}\in T(j)}$, and therefore, $\{Y_j\}_{j\in\N}$ satisfies $\rho_{\R,Y}(\cdot)\leq\rho_{\R,X}(\cdot)$ and thus $\lim\limits_{\tau\to\infty}\rho_{\R,Y}(\tau)<1$. Analogously, since $T(j)\cong \N^{d-1}$, the random field $\zeta^{(j)}=\{X_\mbf{k}:\, \mbf{k}\in T(j)\}$ can be viewed as a $(d-1)$-parameter field with $\rho_{\R,\zeta^{(j)}}(\cdot)\leq \rho_{\R,X}(\cdot)$ (and thus $\lim\limits_{\tau\to\infty}\rho_{\R,\zeta^{(j)}}(\tau)<1$). It holds that $\sum_{\mbf{k}\in S}X_\mbf{k}=\sum_{j\in N(S)}Y_j$ and $Y_j=\sum_{\mbf{k}\in S(j)}\zeta^{(j)}_\mbf{k}$. Now, applying the induction hypothesis first to $\{Y_j\}_{j\in\N}$ and then to $\zeta^{(j)}=\{X_\mbf{k}\}_{\mbf{k}\in T(j)}$, we obtain 
\begingroup
\allowdisplaybreaks
\begin{align*}
\mathrm{E}\left\|\sum_{\mbf{k}\in S}X_\mbf{k}\right\|^r &= \mathrm{E}\left\|\sum_{j\in N(S)}Y_j\right\|^r \\
&\leq B_{1,r}\left\{\sum_{j\in N(S)}\mathrm{E}\|Y_j\|^r + \left(\sum_{j\in N(S)}\mathrm{E}\|Y_j\|^2\right)^{r/2}\right\} \\
&= B_{1,r} \sum_{j\in N(S)}\mathrm{E}\left\|\sum_{\mbf{k}\in S(j)}\zeta^{(j)}_\mbf{k}\right\|^r + B_{1,r} \left(\sum_{j\in N(S)}\mathrm{E}\left\|\sum_{\mbf{k}\in S(j)}\zeta^{(j)}_\mbf{k}\right\|^2\right)^{r/2} \\
&\leq B_{1,r} \sum_{j\in N(S)}B_{d-1,r}\left\{\sum_{\mbf{k}\in S(j)}\mathrm{E}\left\|\zeta^{(j)}_\mbf{k}\right\|^r + \left(\sum_{\mbf{k}\in S(j)}\mathrm{E}\left\|\zeta^{(j)}_\mbf{k}\right\|^2 \right)^{r/2}\right\} \\
&+ B_{1,r} \left[\sum_{j\in N(S)} B_{d-1,2} \left\{\sum_{\mbf{k}\in S(j)}\mathrm{E}\left\|\zeta^{(j)}_\mbf{k}\right\|^2 + \left(\sum_{\mbf{k}\in S(j)}\mathrm{E}\left\|\zeta^{(j)}_\mbf{k}\right\|^2\right)^{2/2}\right\}\right]^{r/2} \\
&= B_{1,r}B_{d-1,r}\left\{\sum_{\mbf{k}\in S}\mathrm{E}\|X_\mbf{k}\|^r + \sum_{j\in N(S)}\left(\sum_{\mbf{k}\in S(j)}\mathrm{E}\|X_\mbf{k}\|^2\right)^{r/2} \right\}\\
&+ 2^{r/2}B_{1,r} B_{d-1,2}^{r/2}\left(\sum_{\mbf{k}\in S}\mathrm{E}\|X_\mbf{k}\|^2\right)^{r/2} \\
&\leq (B_{1,r}B_{d-1,r} + 2^{r/2} B_{d-1,2}^{r/2}B_{1,r})\left\{ \sum_{\mbf{k}\in S}\mathrm{E}\|X_\mbf{k}\|^r+\left(\sum_{\mbf{k}\in S}\mathrm{E}\|X_\mbf{k}\|^2\right)^{r/2}\right\},
\end{align*}%
\endgroup
where the well-known inequality $\left(\sum_{k=1}^m a_k\right)^q \geq \sum_{k=1}^m a_k^q$ 
is used for the last inequality.

\eqref{eq:momIneq} is a trivial consequence of \eqref{eq:rosenthalIneq}. If \eqref{eq:momIneq} holds for some $r>2$ and every block in $\Z^d$, \eqref{eq:maxIneq} follows from Corollary 1 in \cite{moricz1983} (see also \cite{bulinskiShashkin}, Chapter 2, Theorem 1.2, p. 108). (The Corollary can be applied in any normed space without changing the proof.)
\end{proof}

Following an approach that is similar to \cite{davidson2002} (see Theorems 29.6 and 29.18), we aim to reduce the multivariate FCLT to the corresponding results for the univariate case. For real-valued processes, Deo \cite{deoFCLT1975} gave a version for random fields of Theorems 19.1 and 19.2 of \cite{billingsley1968}, which use a characterization of Brownian motion to obtain a general FCLT (see Lemmas 2 and 3 in \cite{deoFCLT1975}). We extend this result to multivariate random fields by taking advantage of the fact that Gaussian random vectors can be characterized by their behavior under projections.
\begin{lemma}\label{le:FCLTRk}
Let $\Sigma$ be a symmetric positive semidefinite matrix and $S_n=\{S_n(\mbf{t})\}_{\mbf{t}\in[0,1]^d}$ a sequence of stochastic processes with sample paths in $D_{\R^k}([0,1]^d)$, such that 
\begin{enumerate}
	\item[(i)] $\mathrm{E}S_n(\mbf{t})\to \undnr{0}$ and $\Cov S_n(\mbf{t}) \to [\mbf{t}] \Sigma$ as $n\to\infty$, for each $\mbf{t}\in[0,1]^d$, 
	\item[(ii)] the set $\{\|S_n(\mbf{t})\|^2\}_{n}$ is uniformly integrable for each $\mbf{t}$,
	\item[(iii)] if $B_1,\ldots,B_p$ is a collection of strongly separated blocks, then the increments $S_n(B_1)$,\ldots, $S_n(B_p)$ are asymptotically independent in the sense that if $H_1$,\ldots,$H_p$ are arbitrary Borel sets in $\R^k$, then the difference
	\begin{align*}
	&\Pr(S_n(B_1)\in H_1,\ldots,S_n(B_p)\in H_p) - \prod_{i=1}^p \Pr(S_n(B_i)\in H_i)
	\end{align*}
	goes to zero as $n\to\infty$ and,
	\item[(iv)] for each $\varepsilon>0$, $\eta>0$, we can find a $\delta>0$ such that 
	\[\Pr(w^k(S_n,\delta)>\varepsilon)<\eta\]
	for all sufficiently large $n$, where we define the modulus of continuity
\[w^k(x;\delta):=\sup\{\|x(\mbf{t})-x(\mbf{s})\|: \|\mbf{t}-\mbf{s}\|\leq \delta\}, \]
for $x\in D_{\R^k}([0,1]^d)$ and $0<\delta<1$.  
\end{enumerate}
Then $S_n$ converges weakly in $D_{\R^k}([0,1]^d)$ to the $k$-dimensional Brownian sheet on $[0,1]^d$ with covariance matrix $\Sigma$.
\end{lemma}
\begin{proof}
Consider $\blambda\in\R^k$ and define $\{S_n^{\blambda}(\mbf{t})\}_{\mbf{t}\in[0,1]^d}$ by $S_n^{\blambda} (\mbf{t})=\blambda^\top S_n(\mbf{t})$. First, note that for any $x\in D_{\R^k}([0,1]^d)$, $\mbf{t}\in[0,1]^d$ and $\blambda\in\R^k$, it holds that if $Q(\mbf{t})$ is some quadrant in $[0,1]^d$, then $x_{Q(\mbf{t})}(\mbf{t})=\lim_{\mbf{s}\to\mbf{t},\mbf{s}\in Q(\mbf{t})} x(\mbf{s})$ exists, and since $\mbf{y} \mapsto \blambda^\top \mbf{y}$ is a continuous map, it follows that 
\[\lim_{\mbf{s}\to\mbf{t},\mbf{s}\in Q(\mbf{t})}\blambda^\top x(\mbf{s})=\blambda^\top \lim_{\mbf{s}\to\mbf{t},\mbf{s}\in Q(\mbf{t})} x(\mbf{s})= \blambda^\top x_{Q(\mbf{t})}(\mbf{t})\]
also exists. Therefore, if $x\in D_{\R^k}([0,1]^d)$, then $\blambda^\top x \in D_{\R}([0,1]^d)$ and if $x \in C_{\R^k}([0,1]^d)$, then $\blambda^\top x \in C_{\R}([0,1]^d)$. Furthermore, since $D_{\R^k}([0,1]^d)\to D_{\R}([0,1]^d)$, $x\mapsto \blambda^\top x$, is a continuous map, $S_n^{\blambda}$ are random elements in $D_{\R}([0,1]^d)$. Assumptions $(i)-(iii)$ imply:
\begin{itemize}
	\item[(i)] $\mathrm{E}S_n^{\blambda}(\mbf{t})=\blambda^\top \mathrm{E}S_n(\mbf{t})\to 0$, $\Cov\left(\blambda^\top S_n(\mbf{t})\right)=\blambda^\top \Cov\left(S_n(\mbf{t})\right)\blambda\to [\mbf{t}] \blambda^\top \Sigma \blambda$ for any $\mbf{t}\in [0,1]^d$.
	\item[(ii)] $\{|S_n^{\blambda}(\mbf{t})|^2\}_{n\geq 1}$ is uniformly integrable for each $\mbf{t}$, since due to the Cauchy-Schwarz inequality $\left| \blambda^\top S_n(\mbf{t})\right|^2 \leq \|\blambda\|^2 \|S_n(\mbf{t})\|^2$.
	\item[(iii)] For arbitrary linear Borel sets $H_1,\ldots,H_p$, the sets $f^{-1}_{\blambda}(H_1),\ldots,f^{-1}_{\blambda}(H_p)$ where $f_{\blambda}$ is the continuous map $f_{\blambda}:\R^k\to\R$, $\mbf{x} \mapsto \blambda^\top \mbf{x}$, lie in $\mathcal{B}(\R^k)$. Therefore, for any collection of strongly separated blocks $B_1,\ldots,B_p$,
		\begin{align*}
	&\Pr(S_n^{\blambda}(B_1)\in H_1,\ldots,S_n^{\blambda}(B_p)\in H_p) - \prod_{i=1}^p \Pr(S_n^{\blambda}(B_i)\in H_i) \\
	=& \Pr\left(S_n(B_1)\in f^{-1}_{\blambda}(H_1),\ldots,S_n(B_p)\in f^{-1}_{\blambda}(H_p)\right) - \prod_{i=1}^p \Pr\left(S_n(B_i)\in f^{-1}_{\blambda}(H_i)\right) 
	\end{align*}
	goes to zero as $n\to\infty$ by assumption $(iii)$.
	\item[(iv)] Since by the Cauchy-Schwarz inequality 
	\[|\blambda^\top\left\{S_n(\mbf{t})-S_n(\mbf{s})\right\}|\leq \|\blambda\|\|S_n(\mbf{t})-S_n(\mbf{s})\|,\]
	it trivially holds that:
	\[\forall \varepsilon>0, \eta>0\,\exists \delta>0:\; \Pr\left(\omega^1(S_n^{\blambda},\delta)>\varepsilon \|\blambda\|\right) \leq \Pr\left(\omega^k(S_n,\delta)>\varepsilon\right) < \eta\]
\end{itemize}
Therefore, if $\blambda^\top\Sigma\blambda>0$, $(\blambda^\top\Sigma\blambda)^{-1/2}S_n^{\blambda}$ fulfills the conditions of Lemma 3 in \cite{deoFCLT1975} and thus converges to a standardized Brownian sheet $(\blambda^\top\Sigma\blambda)^{-1/2} W^{\blambda}$ in $D_{\R}([0,1]^d)$. By continuous mapping, this implies $S_n^{\blambda} \Rightarrow W^{\blambda}$ in $D_{\R}([0,1]^d)$. If $\blambda^\top\Sigma\blambda=0$, the processes $S_n^{\blambda}\equiv 0$ and $W^{\blambda} \equiv 0$ are both degenerated and therefore $S_n^{\blambda} \Rightarrow W^{\blambda}$ holds trivially.\\
In particular, every coordinate process $S_n^{i}=S_n^{e_i}$ (where $e_i\in\R^k$ is the vector with one in position $i$ and zero elsewhere ($i\in\{1,\ldots,k\}$)) is tight in $D_{\R}([0,1]^d)$ and thus for any $\varepsilon>0$, we can find $M_\varepsilon\in(0,\infty)$ such that
\[\Pr(\|S_n\|_\infty>M_\varepsilon )\leq \sum_{i=1}^k \Pr(\|S_n^{i}\|_\infty>M_\varepsilon)\leq \varepsilon\quad\forall\,n\in\N.\]
Therefore, assumption $(iv)$ implies that $S_n$ is tight in $D_{\R^k}([0,1]^d)$. 
Now, consider a convergent subsequence $S_{n'}$, say $S_{n'}\Rightarrow W$. Then the continuity of the mappings $D_{\R^k}([0,1]^d)\to D_{\R}([0,1]^d)$, $x\mapsto \blambda^\top x$, for any $\blambda\in\R^k$ implies $S_{n'}^{\blambda}=\blambda^\top S_{n'} \Rightarrow \blambda^\top W = W^{\blambda}$, where $W^{\blambda}$ is a Brownian sheet in $D_\R([0,1]^d)$ with covariance $\blambda^\top \Sigma \blambda\geq 0$. In order to show that $S_n$ converges in $D_{\R^k}([0,1]^d)$, it suffices to show that $W$ (and therefore any limit of a convergent subsequence) is indeed the Brownian sheet in $H=\R^k$. Denote the coordinate processes by $W^{i}=W^{e_i}$. 
Since this holds for all the coordinate processes, $W$ is a.s. continuous and $W(\mbf{t})=\undnr{0}$ a.s. for any $\mbf{t}\in[0,1]^d$ with $[\mbf{t}]=0$.\\
From $W\in C_{\R^k}([0,1]^d)$ a.s., it follows that the projection maps $\pi_{\mbf{t}_1,\ldots,\mbf{t}_l}$ are $\Pr_W$-a.s. continuous and therefore $S_{n'}\Rightarrow W$ implies the convergence of the finite dimensional distributions. Since for a block $B=(\mbf{s},\mbf{t}]$, 
\[S_{n'}(B)=\sum\limits_{\boldsymbol{\varepsilon}\in\{0,1\}^d}(-1)^{d-\sum_{j=1}^d \varepsilon_j} S_{n'}\left(\mbf{s}+\boldsymbol{\varepsilon}(\mbf{t}-\mbf{s})\right),\] this implies the weak convergence of the increments of $S_{n'}$.
The increments $W(B)$ of $W$ have a Gaussian distribution with mean zero and covariance $\lambda(B)\Sigma$ in $\R^k$, since $W(B)=\left(W^{1}(B),\ldots,W^{k}(B)\right)^\top$ and $\sum_{i=1}^{k}\lambda_i W^{i}(B)=W^{\blambda}(B)$ is a centered Gaussian random variable with variance $\lambda(B)\blambda^\top \Sigma \blambda$ for any $\blambda\in\R^k$. 
In particular, the distribution of $W(B)$ is absolutely continuous, so that for any collection of strongly separated blocks $B_1$,\ldots,$B_p$ and any $y_1,\ldots,y_p\in\R^k$, we have 
\[\Pr(S_{n'}(B_j)\leq y_j)\to \Pr(W(B_j)\leq y_j)\quad \text{($j=1,\ldots,p$)}\]
and therefore
\begin{align*}
&\Pr(W(B_1)\leq y_1,\ldots, W(B_p)\leq y_p) \\
=& \lim_{n'\to\infty}\Pr(S_{n'}(B_1)\leq y_1,\ldots, S_{n'}(B_p)\leq y_p) \\
=&\lim_{n'\to\infty}\left\{\Pr(S_{n'}(B_1)\leq y_1,\ldots, S_{n'}(B_p)\leq y_p)- \prod_{j=1}^p \Pr(S_{n'}(B_j)\leq y_j)\right\}\\
& + \prod_{j=1}^p \lim_{n'\to\infty} \Pr(S_{n'}(B_j)\leq y_j) \\
\stackrel{(iii)}{=}&\prod_{j=1}^p \Pr(W(B_j)\leq y_j).
\end{align*}
Note that due to the a.s. continuity of $W$, this also yields the independence of the increments over any (not necessarily strongly separated) collection of pairwise disjoint blocks.
\end{proof}

\begin{lemma}\label{le:FCLTRkMixing}
Let $\{X_{\mbf{j}}\}_{\mbf{j}\in \Z^d}$ be an $\R^k$-valued $\rho_{\R}$-mixing, weakly stationary centered random field, $\{S_n(\mbf{t})\}_{\mbf{t}\in[0,1]^d}$ a process in $D_{\R^k}([0,1]^d)$ with 
\[S_n(\mbf{t})= n^{-d/2} \sum\limits_{\undnr{1}\leq \mbf{j}\leq \gaus{n\mbf{t}}}X_{\mbf{j}} \] 
and $\Sigma(n,\mbf{t})=\Cov\left(S_n(\mbf{t})\right)$. If
\begin{itemize}
	\item[(i)] $\sup\limits_{\mbf{j}\in \Z^d}\mathrm{E}\|X_{\mbf{j}}\|^{2+\delta}<\infty$ for some $\delta>0$ and
	\item[(ii)] $\sum_{m\geq 1}m^{d-1}\alpha_{1,1}(m)^{\delta/(2+\delta)}<\infty$,
\end{itemize} 
then $\Sigma(n,\mbf{t})\rightarrow [\mbf{t}]\Sigma$ for any $\mbf{t}\in[0,1]^d$ and a positive semidefinite matrix $\Sigma=(\sigma_{i,j})_{1\leq i,j\leq k}$ with $\sigma_{i,j}=\sum_{\mbf{v}\in\Z^d}\gamma_{i,j}(\mbf{v})$, where $\gamma_{i,j}(\mbf{v})=\Cov(X^{i}_{\undnr{0}},X^{j}_{\mbf{v}})$, and the series converges absolutely. 
Furthermore, $\{S_n(\mbf{t})\}_{\mbf{t}\in[0,1]^d}$ converges in $D_{\R^k}([0,1]^d)$ to a $k$-dimensional Brownian sheet with covariance matrix $\Sigma$.
\end{lemma}
\begin{proof}
As remarked by Guyon \cite{guyon1995} (p. 110), for any $i,j\in\{1,\ldots,k\}$ the covariance inequality (see \cite{rio2013inequalities}, Theorem 1.1, and note that there is an additional factor 2 in Rio's definition of the mixing coefficient)
\[|\gamma_{i,j}(\mbf{v})|=|\Cov(X^{i}_{\undnr{0}},X^{j}_\mbf{v})| \leq 4 \alpha_{1,1}(\|\mbf{v}\|_\infty)^{\delta/(2+\delta)} \|X_{\undnr{0}}\|_{2+\delta}^2 \]
together with assumptions $(i)$ and $(ii)$ implies $\sum_{\mbf{v}\in\Z^d}|\gamma_{i,j}(\mbf{v})|<\infty$.
Using this and the dominated convergence theorem, we obtain 
\begin{align*}
\Sigma(n,\mbf{t})^{(i,j)} &= n^{-d} \Cov\left(\sum\limits_{\undnr{1}\leq\mbf{m}\leq\gaus{n\mbf{t}}}X^{i}_\mbf{m}, \sum\limits_{\undnr{1}\leq\mbf{m}'\leq\gaus{n\mbf{t}}}X^{j}_{\mbf{m}'} \right) \\
&=n^{-d}\sum\limits_{-\gaus{n\mbf{t}}<\mbf{v}<\gaus{n\mbf{t}}}\gamma_{i,j}(\mbf{v})\prod_{l=1}^d(\gaus{nt_l}-|v_l|)\\
&=\sum\limits_{-\undnr{n}\leq\mbf{v}\leq\undnr{n}}I_{\{|\mbf{v}|\leq \gaus{n\mbf{t}}\}}\prod_{l=1}^d \frac{\gaus{nt_l}-|v_l|}{n} \gamma_{i,j}(\mbf{v})\\
&\stackrel{n\to\infty}{\longrightarrow} [\mbf{t}] \sigma_{i,j}
\end{align*}
for any $\mbf{t}\in[0,1]^d$. Furthermore, it follows that the matrix $\Sigma$ is positive semidefinite as the limit of the positive semidefinite covariance matrices $\Sigma(n,\undnr{1})$.

We show that Lemma \ref{le:FCLTRk} can be applied to obtain the stated convergence. First, note that condition $(i)$ of Lemma \ref{le:FCLTRk} is fulfilled, since $\{X_{\mbf{j}}\}_{\mbf{j}\in \Z^d}$ is centered and $\Sigma(n,\mbf{t})\stackrel{n\to\infty}{\longrightarrow} [\mbf{t}]\Sigma$.\\
The assumptions imply the moment inequality \eqref{eq:momIneq} from Lemma \ref{le:RosZhang}. Therefore, condition $(ii)$ follows from 
\[\sup_{n\geq 1}\mathrm{E}\|S_n(\mbf{t})\|^{2+\delta}\leq [\mbf{t}]^{1+\delta/2} C(r,d,X)\leq C(r,d,X)<\infty \]
for any $\mbf{t}$. \\
For strongly separated blocks $B_1=(\mbf{s}_1,\mbf{t}_1],\ldots,B_q=(\mbf{s}_q,\mbf{t}_q]$, there is an $i\in\{1,\ldots,d\}$ such that $0\leq s_1^{i}\leq t_1^{i}< s_2^{i}\leq t_2^{i}<\ldots < s_q^{i}\leq t_q^{i}\leq 1$ (after reordering the blocks if necessary), i.e., $\min\limits_{j=1,\ldots,q-1}(s_{j+1}^{i}-t_j^{i})>0$, and therefore $\min\limits_{j=1,\ldots,q-1}(\gaus{ns_{j+1}^{i}}-\gaus{nt_j^{i}})\to \infty$ for $n\to\infty$. Then
\begingroup
\allowdisplaybreaks
\begin{align*}
& \Pr\left(\bigcap_{j=1}^q \{S_n(B_j)\in H_j\}\right) - \prod_{j=1}^q \Pr(S_n(B_j)\in H_j) \\
=& \Pr\left(\left\{\bigcap_{j=1}^{q-1} \{S_n(B_j)\in H_j\}\right\}\cap\{S_n(B_q)\in H_q\}\right) \\
&- \Pr\left(\bigcap_{j=1}^{q-1} \{S_n(B_j)\in H_j\}\right) \Pr(S_n(B_q)\in H_q)\\
&+ \Pr(S_n(B_q)\in H_q)\left[ \Pr\left(\left\{\bigcap_{j=1}^{q-2} \{S_n(B_j)\in H_j\}\right\}\cap\{S_n(B_{q-1})\in H_{q-1}\}\right)\right.\\
&\left. - \Pr\left(\bigcap_{j=1}^{q-2} \{S_n(B_j)\in H_j\}\right) \Pr(S_n(B_{q-1})\in H_{q-1}) \right]\\
&\resizebox{0.979\linewidth}{!}{$+\Pr(S_n(B_q)\in H_q)\Pr(S_n(B_{q-1})\in H_{q-1})\left[ \Pr\left(\left\{\bigcap_{j=1}^{q-3} \{S_n(B_j)\in H_j\}\right\}\cap \{S_n(B_{q-2})\in H_{q-2}\}\right)\right.$}\\
&\left. - \Pr(S_n(B_1)\in H_1,\ldots, S_n(B_{q-3})\in H_{q-3}) \Pr(S_n(B_{q-2})\in H_{q-2})  \Bigg]\right.\\
&+\ldots + \prod_{j=1}^q \Pr(S_n(B_j)\in H_j) - \prod_{j=1}^q \Pr(S_n(B_j)\in H_j) \\
\leq& q \, \rho_{\R}\left(\min\limits_{j=1,\ldots,q-1}(\gaus{ns_{j+1}^{i}}-\gaus{nt_j^{i}})\right) \stackrel{n\to\infty}{\longrightarrow} 0
\end{align*}%
\endgroup
Thus, condition $(iii)$ of Lemma \ref{le:FCLTRk} is fulfilled. Finally, using (the proof of) Theorem 1.3 in \cite{bulinskiShashkin} (Chapter 5, p. 253), we will now show that condition $(iv)$ of the Lemma is implied by \eqref{eq:momIneq}. As noted in Lemma \ref{le:RosZhang}, \eqref{eq:momIneq} together with assumption $(i)$ imply \eqref{eq:maxIneq} for any block $U$. Analogously to the proof of condition $(ii)$, this implies the uniform integrability of $\{(\#U_n)^{-1} M(U_n)^2\}_{n\geq 1}$ (see \eqref{eq:defMax} for the notation) for any sequence of blocks $U_n$ growing to infinity. The proof of Theorem 1.3 in \cite{bulinskiShashkin} (Chapter 5, p. 253) therefore shows $(iv)$.
\end{proof}

The following corollary of Theorem 4.2 in \cite{billingsley1968} is an adaptation of Lemma 4.1 in \cite{chenwhite1998} to multiparameter processes.
\begin{lemma}\label{le:ChenWhite}
Let $K\in\N$ be fixed and let $\{X_n=(X_{n,1},\ldots,X_{n,K}):\; n\geq 1\}$ be a sequence of $D_H([0,1]^d)^K$-random elements. Let $X_1^k$,\ldots,$X_K^k$ be independent $d$-parameter Brownian motions in $H_k$ with $\mathrm{E}X_i^k(\undnr{1})=0$ and $\Cov X_i^k(\undnr{1})=S_i^k$, $i=1,\ldots,K$. Suppose the following conditions hold:
\begin{itemize}
	\item[(a)] For each $k\geq 1$, $(P_kX_{n,1},\ldots P_k X_{n,K}) \Rightarrow (X_1^k,\ldots,X_K^k)$ in $D_{H_k}([0,1]^d)^K$ as $n\to\infty$;
	\item[(b)] $(X_1^k,\ldots,X_K^k) \Rightarrow (X_1,\ldots,X_K)$ in $D_{H}([0,1]^d)^K$ as $k\to\infty$;
	\item[(c)] \begin{align*}&\limsup_{n\to\infty}\mathrm{E}\left[ \sup_{\mbf{t}_1,\ldots,\mbf{t}_K\in [0,1]^d}\left\|\left(X_{n,1}(\mbf{t}_1),\ldots,X_{n,K}(\mbf{t}_K)\right)-\left(P_k X_{n,1}(\mbf{t}_1),\ldots,P_k X_{n,K}(\mbf{t}_K)\right)\right\|^r\right] \\
	&\longrightarrow 0, \quad\text{ as $k\to\infty$ for some $r\geq 2$.}\end{align*}
\end{itemize}
Then $(X_{n,1},\ldots,X_{n,K})\Rightarrow (X_1,\ldots,X_K)$ in $D_{H}([0,1]^d)^K$, where $X_i$ are independent $d$-parameter Brownian motions in $H$ with $\mathrm{E}X_i(\undnr{1})=0$ and $\Cov X_i(\undnr{1})=S_i$, $i=1,\ldots,K$.
\end{lemma}
Now, we give some preliminary results needed for the proof of Theorem \ref{theo2}. In the next two lemmas, we will establish a Rosenthal inequality for the bootstrapped partial sum process.

\begin{lemma}\label{lemboot1} Let $X,Y$ be random variables taking values in a Hilbert space $H_1$, $X$ is $\mathcal{F}$-measurable and  $Y$ is $\mathcal{G}$-measurable. Let $V$ be a random variable which is independent of $\sigma(\mathcal{F},\mathcal{G})$ and takes values in a Hilbert space $H_2$. Furthermore, let $g,h:H_1\times H_2\rightarrow H$ be measurable functions with
\begin{equation*}
\mathrm{E}\left[g(X,V)\big| V\right]=\mathrm{E}\left[h(Y,V)\big| V\right]=0 \quad a.s.
\end{equation*}
If $\rho=\rho_{\R}(\mathcal{F},\mathcal{G})<1$, then for any $p>1$ such that $\mathrm{E}[\|g(X,V)\|^p] < \infty$ and  $\mathrm{E}[\|h(Y,V)\|^p]<\infty$, there exists a constant $C_{\rho,p}$ such that
\begin{equation*}
\mathrm{E}\left[\big\|g(X,V)\big\|^p\right]\leq C_{\rho,p}\mathrm{E}\left[\big\|g(X,V)+h(Y,V)\big\|^p\right].
\end{equation*}
\end{lemma}
\begin{proof}
We will make use of the conditional expectations 
\[\mathrm{E}\left[\big\|g(X,V)\big\|^p\big| V=v\right]=\mathrm{E}[\|g(X,v)\|^p] \quad\text{ and }\quad \mathrm{E}\left[\big\|h(Y,V)\big\|^p\big| V=v\right]=\mathrm{E}[\|h(Y,v)\|^p].\]
$g(X,v)$ and $h(Y,v)$ are $H$-valued random variables which are $\mathcal{F}$- and $\mathcal{G}$-measurable, respectively. So we can apply Theorem 1 of \cite{zhang1998} to the conditional expectations and obtain
\begin{equation*}
\mathrm{E}\left[\big\|g(X,V)\big\|^p\big| V=v\right]\leq C_{\rho,p}\mathrm{E}\left[\big\|g(X,V)+h(Y,V)\big\|^p\big| V=v\right]
\end{equation*}
and consequently
\begin{align*}
\mathrm{E}\left[\big\|g(X,V)\big\|^p\right]&=\mathrm{E}\left[\mathrm{E}\left[\big\|g(X,V)\big\|^p\big| V\right]\right]\\
&\leq \mathrm{E}\left[C_{\rho,p}\mathrm{E}\left[\big\|g(X,V)+h(Y,V)\big\|^p\big| V\right]\right]\\
&=C_{\rho,p}\mathrm{E}\left[\big\|g(X,V)+h(Y,V)\big\|^p\right].
\end{align*}
\end{proof}

\begin{lemma}\label{lemboot2} Under the assumptions of Theorem \ref{theo2}, for any $r\geq 2$ there exists a constant $B_{d,r}$ such that for any finite subset $S\subset \mathbb{N}^d$ and $i\in\{1,\ldots,K\}$
\begin{align*}
&\mathrm{E}\Big\|\sum_{\mathbf{k}\in S^{(n)}}(X_\mbf{k}-\mu)V_{n,i}(\mathbf{k})\Big\|^r\\
\leq& B_{d,r}\left\{\sum_{\mathbf{k}\in S^{(n)}}\mathrm{E}\left\|(X_\mbf{k}-\mu)\right\|^r \mathrm{E}\left\|V_{n,i}(\mathbf{k})\right\|^r+\Big(\sum_{\mathbf{k}\in S^{(n)}}\mathrm{E}\left\|(X_\mbf{k}-\mu)\right\|^2 \mathrm{E}\left\|V_{n,i}(\mathbf{k})\right\|^2\Big)^{r/2}\right\},
\end{align*}
where $S^{(n)}=S\cap\{1,\ldots,n\}$.
For any block $U\subseteq\{1,\ldots,n\}^d$ and 
\[M^\star(U)=\max\limits_{W\triangleleft U}\left\|\sum\limits_{\mbf{j}\in W}(X_\mbf{j}-\mu)V_{n,i}(\mbf{j})\right\|, \]
it then holds that 
\[
\mathrm{E}\left[M^\star(U)^r\right] \leq C_r (\#U)^{r/2}
\]
for $r\in(2,2+\delta]$ and some $C_r>0$ that may depend on $r$ but not on $U$ or $n$.
\end{lemma}

\begin{proof} This inequality follows in the same way as Theorem 2 of \cite{zhang1998} and Lemma \ref{le:RosZhang} above, using Lemma \ref{lemboot1} instead of Theorem 1 of \cite{zhang1998}.
\end{proof}

\begin{lemma}\label{rem:LRV} Under the assumptions of Theorem \ref{theo2}, for any $B\subseteq (0,1]^d$ which is either a block or a finite union of disjoint blocks, we have 
\begin{equation*}
\hat{\Sigma}_n(B):=\sum_{\mbf{h}\in B_{n}\ominus B_{n}}\omega\left(\mbf{h}/q\right)\frac{1}{n^d}\sum_{\mbf{a}:\,\mathbf{a},\mbf{a}+\mbf{h}\in B_{n}}\left\{X^{(k)}_{\mathbf{a}}-\hat{\mu}^{(k)}(\mbf{a})\right\}\left\{X^{(k)}_{\mathbf{a}+\mbf{h}}-\hat{\mu}^{(k)}(\mbf{a}+\mbf{h})\right\}^\top\xrightarrow{\Pr}\lambda(B)\Sigma,
\end{equation*}
where $\Sigma$ is the long-run variance matrix of $X^{(k)}$, $k\in\N$.
\end{lemma}

\begin{proof} Consider the centered process $\{Y^{(k)}_\mbf{j}\}_{\mbf{j}\in\Z^d}$ with $Y^{(k)}_\mbf{j}=X^{(k)}_\mbf{j}-\mu^{(k)}$. As shown in \cite{bHeuser}, $\left|\hat{\Sigma}_n(B)-\hat{\Sigma}_{Y,n}(B)\right|\xrightarrow{\Pr}0$, where 
\[\hat{\Sigma}_{Y,n}(B) = \sum_{\mbf{h}\in B_{n}\ominus B_{n}}\omega\left(\mbf{h}/q\right)\frac{1}{n^d}\sum_{\mbf{a}:\,\mathbf{a},\mbf{a}+\mbf{h}\in B_{n}}Y^{(k)}_{\mathbf{a}} Y^{(k)\top}_{\mathbf{a}+\mbf{h}}.\]
To obtain the stated convergence, it therefore suffices to show that 
\[\text{(i)}\  \rm{E}\left[\hat{\Sigma}_{Y,n}(B)\right]\longrightarrow \lambda(B)\Sigma \quad\text{and}\quad (ii)\ \mathrm{E}\left[\left(\hat{\Sigma}_{Y,n}(B)-\mathrm{E}\left[\hat{\Sigma}_{Y,n}(B)\right]\right)^2\right]\longrightarrow 0.\]
Only a slight modification of the proof by Lavancier \cite{lavancier2008} (who considered $B=(0,1]^d$ and slightly less general kernel functions $\omega$) is needed to obtain (i). For (ii), we concentrate on the case $k=1$ to simplify notation, but the cases $k\geq 2$ work the same way. Note that by assumption \eqref{eq:assAlpha2} (see \cite{guyon1995}, p. 110), there exists a $C>0$ such that
\begin{align*}
 &\mathrm{E}\left[\left\{\frac{1}{n^d} \sum_{\mbf{a}:\,\mathbf{a},\mbf{a}+\mbf{h}\in B_{n}}\left(Y^{(1)}_{\mathbf{a}}Y^{(1)}_{\mathbf{a}+\mbf{h}}-\mathrm{E}\left[Y^{(1)}_{\mathbf{a}}Y^{(1)}_{\mathbf{a}+\mbf{h}}\right]\right)\right\}^2\right] \\
=& \frac{1}{n^{2d}} \sum_{\mbf{a},\mbf{a}':\,\mathbf{a},\mbf{a}',\mbf{a}+\mbf{h},\mbf{a}'+\mbf{h}\in B_{n}}\Cov\left(Y^{(1)}_{\mathbf{a}}Y^{(1)}_{\mathbf{a}+\mbf{h}},Y^{(1)}_{\mathbf{a}'}Y^{(1)}_{\mathbf{a}'+\mbf{h}}\right)\\
\leq& \frac{1}{n^d} \sum_{\mbf{l}\in\Z^d} \left| \Cov\left(Y^{(1)}_{\undnr{0}}Y^{(1)}_{\mbf{h}},Y^{(1)}_{\mbf{l}}Y^{(1)}_{\mbf{l}+\mbf{h}}\right)\right| \leq C \frac{1}{n^d}.
\end{align*}
Therefore
\begin{align*}
& \mathrm{E}\left[\left(\hat{\Sigma}_{Y,n}(B)-\mathrm{E}\left[\hat{\Sigma}_{Y,n}(B)\right]\right)^2\right] \\
\leq& \left(\sum_{\mbf{h}\in B_{n}\ominus B_{n}} \left|\omega\left(\mbf{h}/q\right)\right| \left\|\frac{1}{n^d} \sum_{\mbf{a}:\, \mathbf{a},\mbf{a}+\mbf{h}\in B_{n}}\left(Y^{(1)}_{\mathbf{a}}Y^{(1)}_{\mathbf{a}+\mbf{h}}-\mathrm{E}\left[Y^{(1)}_{\mathbf{a}}Y^{(1)}_{\mathbf{a}+\mbf{h}}\right]\right)\right\|_2\right)^2 \\
\leq& C \frac{1}{n^d} \left(\sum\limits_{-\undnr{n}\leq\mbf{j}\leq \undnr{n}}|\omega(\mbf{j}/q)| \right)^2 \leq C \frac{q^{2d}}{n^d}\longrightarrow 0.
\end{align*}
\end{proof}

\end{subsection}

\begin{subsection}{Proofs of the main results}

\begin{proof}[Proof of Theorem \ref{thm:FCLT}]
We assume without loss of generality that $\mu=0$ and proceed as in the proof of Theorem 1 in \cite{sharipov2014} by showing the three conditions of Lemma \ref{le:ChenWhite}. First, note that for any $h\in H\setminus \{0\}$, the random field $\{Y_\mbf{j}\}_{\mbf{j}\in\Z^d}$ with $Y_\mbf{j}=\sprod{X_\mbf{j},h}$ is centered, stationary and $\rho_{\R}$-mixing with $\rho_{\R,Y}(x)\leq \rho_{\R,X}(x)$ and $\alpha_{1,1,Y}(x)\leq\alpha_{1,1,X}(x)$, since any $Y_\mbf{j}$ is a measurable transform of $X_\mbf{j}$. Furthermore,
\[\mathrm{E}|Y_\mbf{j}|^{2+\delta} \leq \|h\|^{2+\delta} \mathrm{E}\|X_\mbf{j}\|^{2+\delta}\]
ensures that $\{Y_\mbf{j}\}_{\mbf{j}\in\Z^d}$ has finite $(2+\delta)$-moments. Now, Lemma \ref{le:FCLTRkMixing} implies
\[\left\{\frac{1}{n^{d/2}}\sum\limits_{\undnr{1}\leq \mbf{j}\leq \gaus{n\mbf{t}}} Y_\mbf{j} \right\}_{\mbf{t}\in[0,1]^d} \Rightarrow \{W_h(\mbf{t})\}_{\mbf{t}\in[0,1]^d}, \quad \text{in $D_{\R}([0,1]^d)$,}\]
where $\{W_h(\mbf{t})\}_{\mbf{t}\in[0,1]^d}$ is a Brownian sheet in $\R$ with covariance 
\[\sigma^2(h)=\sum_{\mbf{j}\in\Z^d}\mathrm{E}Y_{\undnr{0}}Y_\mbf{j} = \sum_{\mbf{j}\in\Z^d}\mathrm{E}\left[\sprod{X_{\undnr{0}},h}\sprod{X_\mbf{j},h}\right],\]
and the series converges absolutely. Define the covariance operator $S$ as in \eqref{eq:CovOp}, then $\sprod{Sh,h}=\sigma^2(h)$ holds for all $h\in H\setminus\{0\}$, and $S$ is positive, linear and self-adjoint. Then $S\in\mathcal{S}(H)$, because for any complete orthonormal system $\{e_i\}_{i\in\N}$ in $H$, we obtain
\[\sum_{i=1}^\infty |\sprod{Se_i,e_i}| = \sum_{i=1}^\infty \sprod{Se_i,e_i}= \sum_{i=1}^\infty \lim_{n\to\infty}n^{-d}\mathrm{E}\left[\sum\limits_{\undnr{1}\leq \mbf{j}\leq \undnr{n}}\sprod{X_\mbf{j},e_i}\right]^2 \]
and Theorem 28.10 of \cite{bradley2007} (Volume III, p. 154) implies
\[n^{-d}\mathrm{E}\left[\sum\limits_{\undnr{1}\leq \mbf{j}\leq \undnr{n}}\sprod{X_\mbf{j},e_i}\right]^2\leq C  n^{-d} \sum\limits_{\undnr{1}\leq \mbf{j}\leq \undnr{n}} \mathrm{E}[\sprod{X_\mbf{j},e_i}^2] \leq C \mathrm{E}[\sprod{X_{\undnr{0}},e_i}^2]\]
with a single constant $C$ for all $n$ and $i$. 
Therefore,
\[\sum_{i=1}^\infty |\sprod{Se_i,e_i}| \leq C \sum_{i=1}^\infty \mathrm{E}[\sprod{X_{\undnr{0}},e_i}^2] = C \mathrm{E}\|X_{\undnr{0}}\|^2 < \infty.\]
Define $S_n(\mbf{t})=n^{-d/2}\sum\limits_{\undnr{1}\leq \mbf{k}\leq \gaus{n\mbf{t}}} X_\mbf{k}$ and consider a Brownian sheet $\{W(\mbf{t})\}_{\mbf{t}\in[0,1]^d}$ in $H$ whose covariance operator is defined as in \eqref{eq:CovOp}. Then $\{W^{(k)}(\mbf{t})\}_{\mbf{t}\in[0,1]^d}=\{P_k W(\mbf{t})\}_{\mbf{t}\in[0,1]^d}$ is a Brownian sheet in $H_k$ with covariance operator $S_k=P_k S P_k$. In particular, the covariance operator can be identified with the $k\times k$ nonnegative definite covariance matrix $\Sigma=(\gamma_{i,j})_{1\leq i,j\leq k}$ with $\gamma_{i,j}=\sum_{\mbf{v}\in\Z^d}\mathrm{E}\left[\sprod{X_{\undnr{0}},e_i}\sprod{X_\mbf{v},e_j} \right]$. \\
For each $k\geq 1$, the convergence 
\[\left\{P_k S_n(\mbf{t})\right\}_{\mbf{t}\in[0,1]^d} \Rightarrow \{W^{(k)}(\mbf{t})\}_{\mbf{t}\in[0,1]^d}, \quad \text{in $D_{H_k}([0,1]^d)$,}\]
is equivalent to the FCLT for the $k$-dimensional random field $\tilde X_\mbf{j}^{(k)}=(\sprod{X_\mbf{j},e_1},\ldots,\sprod{X_\mbf{j},e_k})^\top$. Since  $\{\tilde X^{(k)}_\mbf{j}\}_{\mbf{j}\in\Z^d}$ fulfills the assumptions of the Lemma, Lemma \ref{le:FCLTRkMixing} yields
\[\left\{\frac{1}{n^{d/2}}\sum\limits_{\undnr{1}\leq \mbf{j}\leq \gaus{n\mbf{t}}} \tilde X_\mbf{j}^{(k)} \right\}_{\mbf{t}\in[0,1]^d} \Rightarrow \{\tilde W^{(k)}(\mbf{t})\}_{\mbf{t}\in[0,1]^d}, \quad \text{in $D_{\R^k}([0,1]^d)$},\]
where $\{\tilde W^{(k)}(\mbf{t})\}_{\mbf{t}\in[0,1]^d}$ is a Brownian sheet in $\R^k$ with covariance matrix $\Sigma$, i.e., condition $(a)$ of Lemma \ref{le:ChenWhite} is satisfied. \\
Let $\{W(\mbf{t})\}_{\mbf{t}\in[0,1]^d}$ be a Brownian sheet in $H$ with $\Cov W(\mbf{1})=S$, where $S$ is as defined in \eqref{eq:CovOp}. For every $e_i$, $\{\sprod{W(\mbf{t}),e_i}\}_{\mbf{t}\in[0,1]^d}$ is a Brownian sheet in $\R$, and therefore Cairoli's strong inequality (Corollary 2.3.1 in Chapter 7 of \cite{khoshnevisan2002}) for submartingale random fields in $\R$ yields
\begin{align*}
\mathrm{E}\left[\sup_{\mbf{t}\in[0,1]^d}\left\|W(\mbf{t})-W^{(k)}(\mbf{t})\right\|^2\right] &= \mathrm{E}\left[\sup_{\mbf{t}\in[0,1]^d}\sum_{i=k+1}^\infty \sprod{W(\mbf{t}),e_i}^2\right] \\
&\leq \sum_{i=k+1}^\infty \mathrm{E}\left[\sup_{\mbf{t}\in[0,1]^d} \sprod{W(\mbf{t}),e_i}^2\right] \\
&\leq 4^d \sum_{i=k+1}^\infty \mathrm{E}\left[\sprod{W(\undnr{1}),e_i}^2\right] \\
&= 4^d \sum_{i=k+1}^\infty \sprod{Se_i,e_i} \stackrel{k\to\infty}{\longrightarrow} 0,
\end{align*}
which implies $\sup\limits_{\mbf{t}\in[0,1]^d}\|W(\mbf{t})-W^{(k)}(\mbf{t})\|^2 \to 0$ in probability and therefore $W^{(k)} \Rightarrow W$ in $D_H([0,1]^d)$.

Finally, we show condition $(c)$. We note that due to the Hilbert space property,
\[\|A_k(X_{\undnr{1}})\| = \left\| X_{\undnr{1}}-\sum_{i=1}^k \sprod{X_{\undnr{1}},e_i}e_i\right\| \stackrel{k\to\infty}{\longrightarrow} 0\quad a.s.\]
Using $\|A_k(X_{\undnr{1}})\|^r\leq \|X_{\undnr{1}}\|^r$ and the dominated convergence theorem, this implies
\[ \max\{\mathrm{E}\|A_k(X_{\undnr{1}})\|^r,\mathrm{E}\|A_k(X_{\undnr{1}})\|^2\} \stackrel{k\to\infty}{\longrightarrow} 0. \]
We can therefore apply \eqref{eq:maxIneq} to $\{A_k(X_\mbf{j})\}_{\mbf{j}\in\Z^d}$ and obtain
\begin{align*}
&\mathrm{E}\left[\sup_{\mbf{t}\in[0,1]^d}\|S_n(\mbf{t})-P_k S_n(\mbf{t})\|^r \right] \\
=&n^{-rd/2} \mathrm{E}\left[\max\limits_{\undnr{1}\leq \mbf{l}\leq \undnr{n}}\left\|\sum\limits_{\undnr{1}\leq\mbf{j}\leq\mbf{l}}A_k(X_\mbf{j})\right\|^r\right] \\
\leq& \tilde{C} B_{d,r} \left\{\mathrm{E}\|A_k(X_{\undnr{1}})\|^r + \left(\mathrm{E}\|A_k(X_{\undnr{1}})\|^2\right)^{r/2}\right\} \stackrel{k\to\infty}{\longrightarrow} 0,
\end{align*}
where we have used \eqref{eq:maxIneq} for $r=2+\delta>2$. This yields $(c)$ of Lemma \ref{le:ChenWhite}.
\end{proof}

\begin{proof}[Proof of Theorem \ref{theo2}] We will use Lemma \ref{le:ChenWhite}. 
For $k\in\N$, we start by establishing the tightness of $S_{n,1}^{\star(k)},\ldots,S_{n,K}^{\star(k)}$. Since $S_n^{(k)}$ is also tight (this is a direct consequence of the weak convergence of the $k$-dimensional partial sum process, which was proven as part of the proof of Theorem \ref{thm:FCLT}), the tightness of $(S_n^{(k)},S_{n,1}^{\star(k)},\ldots,S_{n,K}^{\star(k)})$ will then follow immediately. 

Note that for any $j\in\{1,\ldots,K\}$ and $\mbf{t}\in[0,1]^d$
\begin{equation*}
S_{n,j}^{\star(k)}(\mathbf{t})=\frac{1}{n^{d/2}}\sum_{\mathbf{1}\leq\mathbf{i}\leq\lfloor n\mathbf{t} \rfloor}\left(X_{\mathbf{i}}^{(k)}-\mu^{(k)}\right)V_{n,j}(\mathbf{i})-\frac{1}{n^{d/2}}\sum_{\mathbf{1}\leq\mathbf{i}\leq\lfloor n\mathbf{t} \rfloor}\left\{\hat{\mu}^{(k)}(\mbf{i})-\mu^{(k)}\right\}V_{n,j}(\mathbf{i}).
\end{equation*}
Using Lemma \ref{lemboot2}, we obtain that the first summand is stochastically bounded and fulfills the tightness condition $(iv)$ of Lemma \ref{le:FCLTRk} (see the proof of Lemma \ref{le:FCLTRkMixing}). Since by assumption the change-set estimator $\hat{C}_n$ is a subblock of $(\undnr{0},\undnr{n}]$, we can bound the second summand by
\begin{align*}
&\left\|\frac{1}{n^{d/2}}\sum_{\undnr{1}\leq\mathbf{i}\leq\lfloor n\mathbf{t} \rfloor}\left\{\hat{\mu}^{(k)}(\mbf{i})-\mu^{(k)}\right\}V_{n,j}(\mathbf{i})\right\|\\
\leq&n^{d/2} \max\limits_{\undnr{1}\leq\mbf{i}\leq\undnr{n}}\left\|\hat{\mu}^{(k)}(\mbf{i})-\mu^{(k)} \right\| \,\frac{1}{n^{d}}\sum_{\undnr{1}\leq\mathbf{i}\leq\lfloor n\mathbf{t} \rfloor}|V_{n,j}(\mathbf{i})| \\
\leq& C \max\limits_{\undnr{1}\leq\mbf{l}<\mbf{m}\leq\undnr{n}}\left\|\frac{1}{n^{d/2}}\sum\limits_{\mbf{l}\leq\mbf{i}\leq\mbf{m}}\left(X^{(k)}_{\mbf{i}} - \mu^{(k)}\right)\right\| \frac{1}{n^{d}}\sum_{\undnr{1}\leq\mathbf{i}\leq\lfloor n\mathbf{t} \rfloor}|V_{n,j}(\mathbf{i})|
\end{align*}
for some $C>0$ and all $\mbf{t}\in[0,1]^d$. By Lemma \ref{le:RosZhang}, the first factor is stochastically bounded. For the second factor, note that due to the Gaussian distribution of $V_{n,j}(\mbf{i})$, for any block $S$ and $r\geq 2$,
\begin{equation}\label{eq:momBoundV}
\mathrm{E}\Big|\sum_{\mathbf{k}\in S}|V_{n,j}(\mathbf{k})|\Big|^r\leq C_r\left(\# S\right)^r
\end{equation}
holds for some constant $C_r >0$. Therefore, the second summand is stochastically bounded. Writing 
\[Y_n(\cdot)=\frac{1}{n^{d/2}}\sum_{\undnr{1}\leq\mathbf{i}\leq\lfloor n\mathbf{\cdot} \rfloor}\left\{\hat{\mu}^{(k)}(\mbf{i})-\mu^{(k)}\right\}V_{n,j}(\mathbf{i}) \quad \text{ and } \quad W_n(\cdot)=n^{-d}\sum_{\undnr{1}\leq \mbf{i}\leq\gaus{n\cdot}}|V_{n,j}(\mbf{i})|,\]
the modulus of continuity of the second summand can be bounded in the following way: 
\begin{align*}
&\Pr\left(\omega^k_{Y_n}(\delta)\geq \varepsilon\right)\\
\leq& \sum_{h=1}^d \Pr\left(\sup\limits_{\mbf{t}\in[0,1]^d:\, t_h \leq 1-\delta, \gamma\in(0,\delta)}\|Y_n(t_1,\ldots,t_{h-1},t_h+\gamma,t_{h+1},\ldots,t_d) - Y_n(\mbf{t})\| \geq \varepsilon d^{-1}\right) \\
=& \sum_{h=1}^d \Pr\left(\sup\limits_{\mbf{t}\in[0,1]^d:\, t_h \leq 1-\delta, \gamma\in(0,\delta)}\frac{1}{n^{d/2}}\left\|\sum\limits_{\mycom{\undnr{1}\leq\mbf{i}\leq\gaus{n\mbf{t}}}{\gaus{nt_h}< i_h \leq \gaus{n (t_h+\gamma)}}}\left\{\hat{\mu}^{(k)}(\mbf{i}) - \mu^{(k)}\right\} V_{n,j}(\mbf{i})\right\| \geq \varepsilon d^{-1}\right) \\
\leq& \sum_{h=1}^d \Pr\left(\max\limits_{\undnr{1}\leq\mbf{i}\leq\undnr{n}}n^{d/2}\|\hat{\mu}^{(k)}(\mbf{i}) - \mu^{(k)}\| \cdot \sup\limits_{\mbf{t}\in[0,1]^d:\, t_h \leq 1-\delta, \gamma\in(0,\delta)}\frac{1}{n^d}\sum\limits_{\mycom{\undnr{1}\leq\mbf{i}\leq\gaus{n\mbf{t}}}{\gaus{nt_h}< i_h \leq \gaus{n (t_h+\gamma)}}}|V_{n,j}(\mbf{i})| \geq \varepsilon d^{-1}\right) \\
\leq& d \Pr\left(\max\limits_{\undnr{1}\leq\mbf{i}\leq\undnr{n}}n^{d/2}\|\hat{\mu}^{(k)}(\mbf{i}) - \mu^{(k)}\| > C\right) \\
&+ \sum_{h=1}^d \Pr\left(\sup\limits_{\mbf{t}\in[0,1]^d:\, t_h \leq 1-\delta, \gamma\in(0,\delta)}|W_n(t_1,\ldots,t_{h-1},t_h+\gamma,t_{h+1},\ldots,t_d) - W_n(\mbf{t})| \geq \varepsilon d^{-1}C^{-1}\right). 
\end{align*}
The first summand goes to 0 uniformly in $n$ for $C\to \infty$. For the second summand, define
\[A_m(h,\delta)=(0,1]\times\ldots\times((m-1)\delta,m \delta\wedge 1]\times\ldots\times(0,1] \]
for $m=1,\ldots,p$ with $p=p(\delta)=\gaus{\delta^{-1}}+1$ and
\[U_{m,n}= \{\gaus{n \mbf{t}}:\, \mbf{t}\in A_m(h,\delta) \}.\]
Then, $\# U_{m,n} \leq n^{d} \delta$, and therefore,
\begin{align*}
&\Pr\left(\sup\limits_{\mbf{t}\in[0,1]^d:\, t_h \leq 1-\delta, \gamma\in(0,\delta)}|W_n(t_1,\ldots,t_{h-1},t_h+\gamma,t_{h+1},\ldots,t_d) - W_n(\mbf{t})| \geq \varepsilon d^{-1}C^{-1}\right) \\
\leq& \sum_{m=1}^p \Pr\left(\sup\limits_{\mbf{s},\mbf{t}\in A_m(h,\delta),\, s_r=t_r \; (r\neq h)}|W_n(\mbf{t}) - W_n(\mbf{s})| \geq \frac{\varepsilon}{2} d^{-1}C^{-1}\right) \\
\leq& \sum_{m=1}^p  \Pr\left(\sup\limits_{V\triangleleft U_{m,n}}n^{-d} \sum_{\mbf{i}\in V} |V_{n,j}(\mbf{i})| \geq \frac{\varepsilon}{4} d^{-1}C^{-1}\right) \\
\leq&  \sum_{m=1}^p  \Pr\left(n^{-d} \sum_{\mbf{i}\in U_{m,n}} |V_{n,j}(\mbf{i})| \geq \frac{\varepsilon}{4} d^{-1}C^{-1}\right) \\
\leq& \sum_{m=1}^p n^{-dr} 4^r d^r C^r \varepsilon^{-r} C_r (\# U_{m,n})^r \\
\leq& 4^r d^r C^r \varepsilon^{-r} C_r (1+\delta^{-1}) \delta \cdot \delta^{r-1} \stackrel{\delta\to 0}{\longrightarrow} 0.
\end{align*}
Thus, condition $(iv)$ of Lemma \ref{le:FCLTRk} is fulfilled for the second summand as well. Therefore, the sum $S_{n,j}^{\star(k)}$ is stochastically bounded with a modulus of continuity that fulfills the tightness condition, and therefore it is tight.

Next, we establish the finite dimensional convergence. Note that due to the tightness of the process, it suffices to show that for any subsequence, there exists a further subsequence such that the finite dimensional distributions converge to the right limit distribution. 
To do this, we first show the following result: For any subsequence $(n_m)_{m\in \mathbb{N}}$, there is another subsequence $(n_m)_{m\in M}$ with $M\subset \mathbb{N}$, such that for all $k,l\in\mathbb{N}$ and all disjoint blocks $B_1,\ldots,B_l$ with corners in $([0,1]\cap\mathbb{Q})^d$, the weak convergence of the conditional (on $X_{\mathbf{i}}$, $\mathbf{i}\leq \undnr{n_m}$) distribution of the random vectors
\begin{equation*}
\mathbf{W}_{m,j}^\star:=\left(S_{n_m,j}^{\star(k)}(B_1),S_{n_m,j}^{\star(k)}(B_2),\ldots,S_{n_m,j}^{\star(k)}(B_l)\right)^\top, \quad j=1,\ldots,K
\end{equation*}
to $\mathbf{W}_j^\star:=\left(W_j^{\star(k)}(B_1),W_j^{\star(k)}(B_2),\ldots,W_j^{\star(k)}(B_l)\right)^\top$, $j=1,\ldots,K$, holds almost surely for $(n_m)_{m\in M}$. \\

To show this, note that conditional on $X_{\mathbf{i}}, \mathbf{i}\leq \undnr{n_m}$, $\mbf{W}_{m,1}^\star,\ldots,\mbf{W}_{m,K}^\star$ are stochastically independent and have a Gaussian distribution with mean 0, so it suffices to show the convergence of the conditional covariance operators. For $j\in\{1,\ldots,K\}$ and $l_1,l_2\in\{1,\ldots,l\}$, the covariance operators are given by 
\begin{align*}
&\Cov^\star\left(S_{n,j}^{\star(k)}(B_{l_1}), S_{n,j}^{\star(k)}(B_{l_2})\right)\\
=&\mathrm{E}\left[S_{n,j}^{\star(k)}(B_{l_1})S_{n,j}^{\star(k)}(B_{l_2})^\top | X_\mbf{i}, \mbf{i}\leq\undnr{n}\right] \\
=&\frac{1}{n^d}\sum_{\mathbf{a}\in B_{l_1,n}}\sum_{\mathbf{b}\in B_{l_2,n}}\{X^{(k)}_{\mathbf{a}}-\hat{\mu}^{(k)}(\mbf{a})\}\{X^{(k)}_{\mathbf{b}}-\hat{\mu}^{(k)}(\mbf{b})\}^\top \mathrm{E}\left[V_{n,j}(\mathbf{a})V_{n,j}(\mathbf{b})\right]\\
=&\sum_{\mbf{h}\in B_{l_2,n}\ominus B_{l_1,n}}\omega\left(\frac{\mbf{h}}{q(n)}\right)\frac{1}{n^d}\sum_{\mbf{a}:\,\mathbf{a}\in B_{l_1,n},\mbf{a}+\mbf{h}\in B_{l_2,n}}\{X^{(k)}_{\mathbf{a}}-\hat{\mu}^{(k)}(\mbf{a})\}\{X^{(k)}_{\mathbf{a}+\mbf{h}}-\hat{\mu}^{(k)}(\mbf{a}+\mbf{h})\}^\top .
\end{align*}
For $l_1=l_2$, this is the covariance estimator proposed by Bucchia and Heuser \cite{bHeuser}. According to Lemma \ref{rem:LRV}, under the assumptions of Theorem \ref{theo2} this estimator converges in probability to $\lambda(B_{l_1})\Sigma$, where $\Sigma$ is the long-run variance matrix. Write $\var^\star\left(S_{n,j}^{\star(k)}(B_{l_1})\right)=\Cov^\star\left(S_{n,j}^{\star(k)}(B_{l_1}),S_{n,j}^{\star(k)}(B_{l_1})\right)$. For $l_1\neq l_2$, it holds that
\begin{align*}
&\var^\star\left(S_{n,j}^{\star(k)}(B_{l_1}\cup B_{l_2})\right) \stackrel{\Pr}{\longrightarrow} \lambda(B_{l_1}\cup B_{l_2}) \Sigma=\left\{\lambda(B_{l_1})+\lambda(B_{l_2})\right\} \Sigma \quad \text{(see Lemma \ref{rem:LRV}),}
\end{align*}
and thus
\begin{align*}
&\Cov^\star\left(S_{n,j}^{\star(k)}(B_{l_1}), S_{n,j}^{\star(k)}(B_{l_2})\right) \\
&= \frac{1}{2}\left\{\var^\star\left(S_{n,j}^{\star(k)}(B_{l_1})+ S_{n,j}^{\star(k)}(B_{l_2})\right)-\var^\star\left(S_{n,j}^{\star(k)}(B_{l_1})\right)-\var^\star\left(S_{n,j}^{\star(k)}(B_{l_2})\right)\right\}\\
&= \frac{1}{2}\left\{\var^\star\left(S_{n,j}^{\star(k)}(B_{l_1}\cup B_{l_2})\right)-\var^\star\left(S_{n,j}^{\star(k)}(B_{l_1})\right)-\var^\star\left(S_{n,j}^{\star(k)}(B_{l_2})\right)\right\}\stackrel{\Pr}{\longrightarrow} 0.
\end{align*}
Therefore, for any subsequence $(n_m)_{m\in\N}$, there exists a further subsequence $(n_m)_{m\in M}$ such that the estimator converges almost surely. Since we only consider countably many blocks $B_i$, by a diagonal sequence argument we can choose a single subsequence $(n_m)_{m\in M}$ so that the almost sure convergence holds for all $k\in\N$ and all blocks with edges in $(\mathbb{Q}\cap[0,1])^d$.

By the following argument, the almost sure weak convergence yields the weak convergence of the joint distribution.
Define the random vectors
\begin{equation*}
\mathbf{W}_m:=\left(S_{n_m}^{(k)}(B_1),S_{n_m}^{(k)}(B_2),\ldots,S_{n_m}^{(k)}(B_l)\right)^\top
\end{equation*}
and $\mathbf{W}:=\left(W^{(k)}(B_1),W^{(k)}(B_2),\ldots,W^{(k)}(B_l)\right)^\top$. Note that by assumption, $\mbf{W},\mbf{W}_1^\star,\ldots,\mbf{W}_K^\star$ are stochastically independent. It holds for any Borel sets $A_0,A_1,\ldots,A_K\subset H_k^l$ that
\begin{align*}
&\left|\Pr\left(\mathbf{W}_m\in A_0,\mathbf{W}_{m,1}^\star\in A_1,\ldots,\mbf{W}_{m,K}^\star\in A_K\right)-\Pr\left(\mathbf{W}\in A_0,\mathbf{W}_1^\star\in A_1,\ldots,\mbf{W}_K^\star\in A_K\right)\right|\\
=&\left|\mathrm{E}\left[\Pr\left(\mathbf{W}_m\in A_0,\mathbf{W}_{m,1}^\star\in A_1,\ldots,\mbf{W}_{m,K}^\star\in A_K\big|\mathbf{X}_{\mathbf{i}}, \mathbf{i}\leq \undnr{n_m}\right)\right] \right. \\
&\left.-\Pr\left(\mathbf{W}\in A_0,\mathbf{W}_1^\star\in A_1,\ldots, \mbf{W}_K^\star \in A_K\right)\right|\\
\leq& \mathrm{E}\left[\mathds{1}_{\{\mathbf{W}_m\in A_0\}}\left|\Pr\left(\mathbf{W}_{m,1}^\star\in A_1,\ldots, \mbf{W}_{m,K}^\star\in A_K\big|\mathbf{X}_{\mathbf{i}}, \mathbf{i}\leq \undnr{n_m}\right)-\Pr\left(\mathbf{W}_1^\star\in A_1,\ldots,\mbf{W}_K^\star\in A_K\right)\right|\right]\\
&+\left|\Pr\left(\mathbf{W}_m\in A_0\right)\Pr\left(\mathbf{W}_1^\star\in A_1,\ldots,\mbf{W}_K^\star\in A_K\right)-\Pr\left(\mathbf{W}\in A_0\right)\Pr\left(\mathbf{W}_1^\star\in A_1,\ldots,\mbf{W}_K^\star \in A_K\right)\right|\\
\leq& \mathrm{E}\left[\left|\Pr\left(\mathbf{W}_{m,1}^\star\in A_1,\ldots,\mbf{W}_{m,K}^\star \in A_K\big|\mathbf{X}_{\mathbf{i}}, \mathbf{i}\leq \undnr{n_m}\right)-\Pr\left(\mathbf{W}_1^\star\in A_1,\ldots, \mbf{W}_K^\star\in A_K\right)\right|\right]\\
&+\left|\Pr\left(\mathbf{W}_m\in A_0\right)-\Pr\left(\mathbf{W}\in A_0\right)\right|\\
&\rightarrow 0,
\end{align*}
as almost sure convergence implies convergence in $L_1$ for bounded random variables and as the last summand converges to 0 by Theorem \ref{thm:FCLT}.

As the process $(S_{n_m}^{(k)},S_{n_m,1}^{\star(k)},\ldots,S_{n_m,K}^{\star(k)})$ is continuous from above, the convergence of all finite dimensional distributions follows from the convergence for all disjoint $B_1,\ldots,B_l$ with corners in $([0,1]\cap\mathbb{Q})^d$ (see the remark after Theorem 3 in \cite{bickel1971}). Together with the tightness of $(S_n^{(k)},S_{n,1}^{\star(k)},\ldots,S_{n,K}^{\star(k)})$, condition (a) of Lemma \ref{le:ChenWhite} follows: for every $k$, the process $(S_n^{(k)},S_{n,1}^{\star(k)},\ldots,S_{n,K}^{\star(k)})$ converges to $(W^{(k)},W_1^{\star(k)},\ldots,W_K^{\star(k)})$.

From the proof of Theorem \ref{thm:FCLT}, we already know that $W^{(k)}\Rightarrow W$ as $k\rightarrow\infty$. $W_1^{\star(k)},\ldots,W_K^{\star(k)}$ and $W_1^\star,\ldots,W_K^{\star(k)}$ are independent copies of $W^{(k)}$ respectively $W$, so condition (b) is obvious.

For condition (c), note that for $r=2+\delta$
\begin{align*}
& \mathrm{E}\left[\sup_{\mathbf{s},\mbf{t}_1,\ldots,\mbf{t}_K\in[0,1]^d}\left\|\left(S_n(\mathbf{s}),S_{n,1}^{\star}(\mathbf{t}_1),\ldots,S_{n,K}^{\star}(\mathbf{t}_K)\right)-\left(S_n^{(k)}(\mathbf{s}),S_{n,1}^{\star(k)}(\mathbf{t}_1),\ldots,S_{n,K}^{\star(k)}(\mathbf{t}_K)\right)\right\|^{r}\right]\\
\leq& 2^{r-1}\mathrm{E}\left[\sup_{\mathbf{s}\in[0,1]^d}\left\|S_n(\mathbf{s})-S_n^{(k)}(\mathbf{s})\right\|^{r}\right] + 2^{K(r-1)}\sum_{j=1}^K \mathrm{E}\left[\sup_{\mathbf{t}\in[0,1]^d}\left\|S_{n,j}^{\star}(\mathbf{t})-S_{n,j}^{\star(k)}(\mathbf{t})\right\|^{r}\right] \\
\leq& 2^{r-1}\mathrm{E}\left[\sup_{\mathbf{s}\in[0,1]^d}\left\|S_n(\mathbf{s})-S_n^{(k)}(\mathbf{s})\right\|^{r}\right] \\
+& 2^{(K+1)(r-1)}\sum_{j=1}^K \mathrm{E}\left[\sup_{\mathbf{t}\in[0,1]^d}\left\|\frac{1}{n^{d/2}}\sum\limits_{\undnr{1}\leq\mbf{i}\leq\gaus{n\mbf{t}}}V_{n,j}(\mbf{i})\left\{A_k(X_\mbf{i})-A_k(\mu)\right\}\right\|^{r}\right] \\
+&2^{(K+1)(r-1)} \sum_{j=1}^K \mathrm{E}\left[\sup_{\mathbf{t}\in[0,1]^d}\left\|\frac{1}{n^{d/2}}\sum\limits_{\undnr{1}\leq\mbf{i}\leq\gaus{n\mbf{t}}}V_{n,j}(\mbf{i})\left\{A_k(\hat{\mu}(\mbf{i}))-A_k(\mu)\right\}\right\|^{r}\right].
\end{align*}
We have already shown in the proof of Theorem \ref{thm:FCLT} that the first term converges to 0 for $k\to\infty$. The convergence to 0 of the second term follows with the same arguments, replacing Lemma \ref{le:RosZhang} with Lemma \ref{lemboot2}. For the third term, consider
\begin{align*}
&\mathrm{E}\left[\sup_{\mathbf{t}\in[0,1]^d}\left\|\frac{1}{n^{d/2}}\sum\limits_{\undnr{1}\leq\mbf{i}\leq\gaus{n\mbf{t}}}V_{n,j}(\mbf{i})\left\{A_k(\hat{\mu}(\mbf{i}))-A_k(\mu)\right\}\right\|^{r}\right] \\
\leq& \mathrm{E}\left[n^{rd/2}\max\limits_{\undnr{1}\leq\mbf{i}\leq\undnr{n}}\left\|A_k\left(\hat{\mu}(\mbf{i})-\mu\right) \right\|^r\right] \cdot \mathrm{E}\left[\left|\frac{1}{n^{d}}\sum\limits_{\undnr{1}\leq\mbf{i}\leq\undnr{n}}|V_{n,j}(\mbf{i})|\right|^{r}\right] \\
\leq& C  \mathrm{E}\left[\max\limits_{\undnr{1}\leq\mbf{l}<\mbf{m}\leq\undnr{n}}\left\|\frac{1}{n^{d/2}}\sum\limits_{\mbf{l}\leq\mbf{i}\leq\mbf{m}}A_k\left(X_{\mbf{i}} - \mu\right)\right\|^r \right]\cdot \mathrm{E}\left[\left|\frac{1}{n^{d}}\sum\limits_{\undnr{1}\leq\mbf{i}\leq\undnr{n}}|V_{n,j}(\mbf{i})|\right|^{r}\right]. 
\end{align*}
Since the first factor goes to 0 (see the proof of Theorem \ref{thm:FCLT}) and the second factor remains bounded (see \eqref{eq:momBoundV}), the proof of condition (c) is finished.
\end{proof}

\end{subsection}
\end{section}

\section*{Acknowledgements}
We are thankful to two anonymous referees for the careful reading of this paper and the very helpful suggestions for improvements. 

\section*{References}
\bibliographystyle{model2-names} 
\bibliography{ReferencesDWB} 
\end{document}